%% file: arxiv-14-8-25.tex
\documentclass{amsart}

\usepackage{amsmath, amsthm}
\usepackage{todonotes}
\usepackage{amsopn,amssymb,eurosym}
\usepackage{amsfonts} 
\usepackage{vmargin}
\usepackage{multirow}
\usepackage{float}

\usepackage{geometry}
\usepackage{booktabs}
\usepackage{algorithm} 
\usepackage{algpseudocode}  
\usepackage{float}

\usepackage[backend=biber,style=alphabetic,url=true,doi=true]{biblatex}
\usepackage{pdfpages,pdflscape}

\usepackage{graphicx}

\usepackage{tcolorbox}
\usepackage{wrapfig}
\usepackage{capt-of} 
\definecolor{fig}{RGB}{76,78,92}
\usepackage[labelfont=bf,font={color=fig,small}]{caption}
\usepackage[colorlinks,citecolor=cyan]{hyperref}
\usepackage{tikz}
\usetikzlibrary{intersections}
\usetikzlibrary{arrows,snakes,shapes}
\usepackage{pgfplots}
\usepgfplotslibrary{fillbetween}
\usepackage{enumitem}
\usepackage{tabularx}
\usepackage{array}
\usepackage{colortbl}
\tcbuselibrary{skins}
\usepackage{setspace}
\usepackage{smartdiagram}
\usepackage{enumitem}
\usepackage{multicol}
\usepackage{wrapfig}
 \usepackage{textcomp}
 \usepackage{etoolbox}
\usepackage{microtype}

\tcbset{tab2/.style={enhanced,fonttitle=\bfseries,fontupper=\normalsize\sffamily,
colback=yellow!10!white,colframe=red!50!black,colbacktitle=red!40!white,
coltitle=black,center title}}

\setmarginsrb{2cm}{2cm} {2cm}{2cm}{0cm}{.5cm}{0cm}{0cm}

\newcommand{\dd}{{\mathrm d}}

\newcommand{\sobh}[1]{{\ensuremath{\operatorname H}^{#1}}}

 \usepackage{color}

\definecolor{amethyst}{rgb}{0.6, 0.4, 0.8}
\definecolor{applegreen}{rgb}{0.55, 0.71, 0.0}
\definecolor{aqua}{rgb}{0.0, 1.0, 1.0}
\definecolor{asparagus}{rgb}{0.53, 0.66, 0.42}
\definecolor{amber(sae/ece)}{rgb}{1.0, 0.49, 0.0}
 	\definecolor{armygreen}{rgb}{0.29, 0.33, 0.13}
	\definecolor{shitbrown}{rgb}{0.43, 0.21, 0.1}
	\definecolor{brightpink}{rgb}{1.0, 0.0, 0.5}
	\definecolor{brightube}{rgb}{0.82, 0.62, 0.91}
	 	\definecolor{byzantine}{rgb}{0.74, 0.2, 0.64}
		\definecolor{chartreuse(web)}{rgb}{0.5, 1.0, 0.0}

\newcommand{\fT}{\stackrel{_\rightarrow}{\mathtt T}}
\newcommand{\bT}{\stackrel{_\leftarrow}{\mathtt T}}

\newcommand{\fL}{\stackrel{_\rightarrow}{\mathtt L}}
\newcommand{\bL}{\stackrel{_\leftarrow}{\mathtt L}}

\newcommand{\fR}{\stackrel{_\rightarrow}{\mathtt R}}
\newcommand{\bR}{\stackrel{_\leftarrow}{\mathtt R}}

\newcommand{\rr}{{\mathtt r}}

\theoremstyle{plain}
\newtheorem{theorem}{Theorem}[section]

\newtheorem{remark}[theorem]{Remark}

\newtheorem{lemma}[theorem]{Lemma}

\newtheorem{corollary}[theorem]{Corollary}

\newcommand{\Norm}[1]{\ensuremath{\left\|#1\right\|}}

\usepackage{xspace}

\renewcommand{\vec}[1]{\ensuremath{\boldsymbol{#1}}}

\input{definitions.tex}

\newcommand{\LL}{\ensuremath{\operatorname L}\xspace}

\renewcommand{\leb}[1]{\ensuremath{\LL^{#1}}}

\newcommand{\ltwop}[2]{\ensuremath{\!\left\langle{#1,#2}\right\rangle}}
\usepackage{mathrsfs}

\newcommand{\cH}{\stackrel{_\rightarrow}{\mathscr H}}
\newcommand{\cbH}{\stackrel{_\leftarrow}{\mathscr H}}

\newcommand{\cS}{\ensuremath{\mathscr S}}
\newcommand{\cV}{\ensuremath{\mathscr V}}

\renewcommand{\d}{\mathop{}\!\mathrm{d}}

\newcommand{\Forall}{\:\forall\:}

\newcommand{\Foreach}{\quad\Forall}

\author{
  Andreas E. Kyprianou
}
\address{
  Andreas Kyprianou
  \thanks{\newline \indent
     Department of Statistics, \newline \indent University of Warwick, \newline \indent Coventry \newline \indent CV4 7AL, UK.
    \newline \indent {\tt{andreas.kyprianou@warwick.ac.uk}}.
}}

\author{
   Aaron Pim
}
\address{
  Aaron Pim, Tristan Pryer
  \thanks{
   \newline \indent Department of Mathematical Sciences, \newline \indent University of Bath, \newline \indent Bath\newline \indent  BA2 7AY, UK.\newline \indent
    {\tt{arp46@bath.ac.uk}, {\tt{tmp38@bath.ac.uk}}}.
}}

\author{
  Tristan Pryer
}

\thanks{ All authors are grateful for  support from the EPSRC programme grant EP/W026899/2 and TP for  support from Leverhulme RPG-2021-238.  }

\begin{document}

\title{A Unified Framework from Boltzmann Transport to Proton Treatment Planning}

\maketitle

\begin{abstract}
  This work develops a rigorous mathematical formulation of proton
  transport by integrating both deterministic and stochastic
  perspectives. The deterministic framework is based on the
  Boltzmann-Fokker-Planck equation, formulated as an operator equation
  in a suitable functional setting. The stochastic approach models
  proton evolution via a track-length parameterised diffusion process,
  whose infinitesimal generator provides an alternative description of
  transport.

  A key result is the duality between the stochastic and deterministic
  formulations, established through the adjoint relationship between
  the transport operator and the stochastic generator. We prove that
  the resolvent of the stochastic process corresponds to the Green’s
  function of the deterministic equation, providing a natural link
  between fluence-based and particle-based transport descriptions.

  The theory is applied to dose computation, where we show that the
  classical relation
  \begin{equation*}
    \text{dose} = \text{fluence} \times \text{mass stopping power}
  \end{equation*}
  arises consistently in both approaches.

  Building on this foundation, we formulate a hybrid optimisation
  framework for treatment planning, in which dose is computed using a
  stochastic model while optimisation proceeds via adjoint-based PDE
  methods. We prove existence and differentiability of the objective
  functional and derive the first-order optimality system. This
  framework bridges stochastic simulation with deterministic control
  theory and provides a foundation for future work in constrained,
  adaptive and uncertainty-aware optimisation in proton therapy.
\end{abstract}

\section{Introduction}
\label{sec:intro}

\subsection{Motivation}

Proton therapy is an advanced form of radiotherapy that exploits the
physical properties of charged particles to deliver highly localised
radiation doses to tumours while minimising collateral damage to
healthy tissue \cite{newhauser2015physics}. 

\begin{wrapfigure}{r}{0.45\linewidth}
  \centering
  \includegraphics[width=\linewidth]{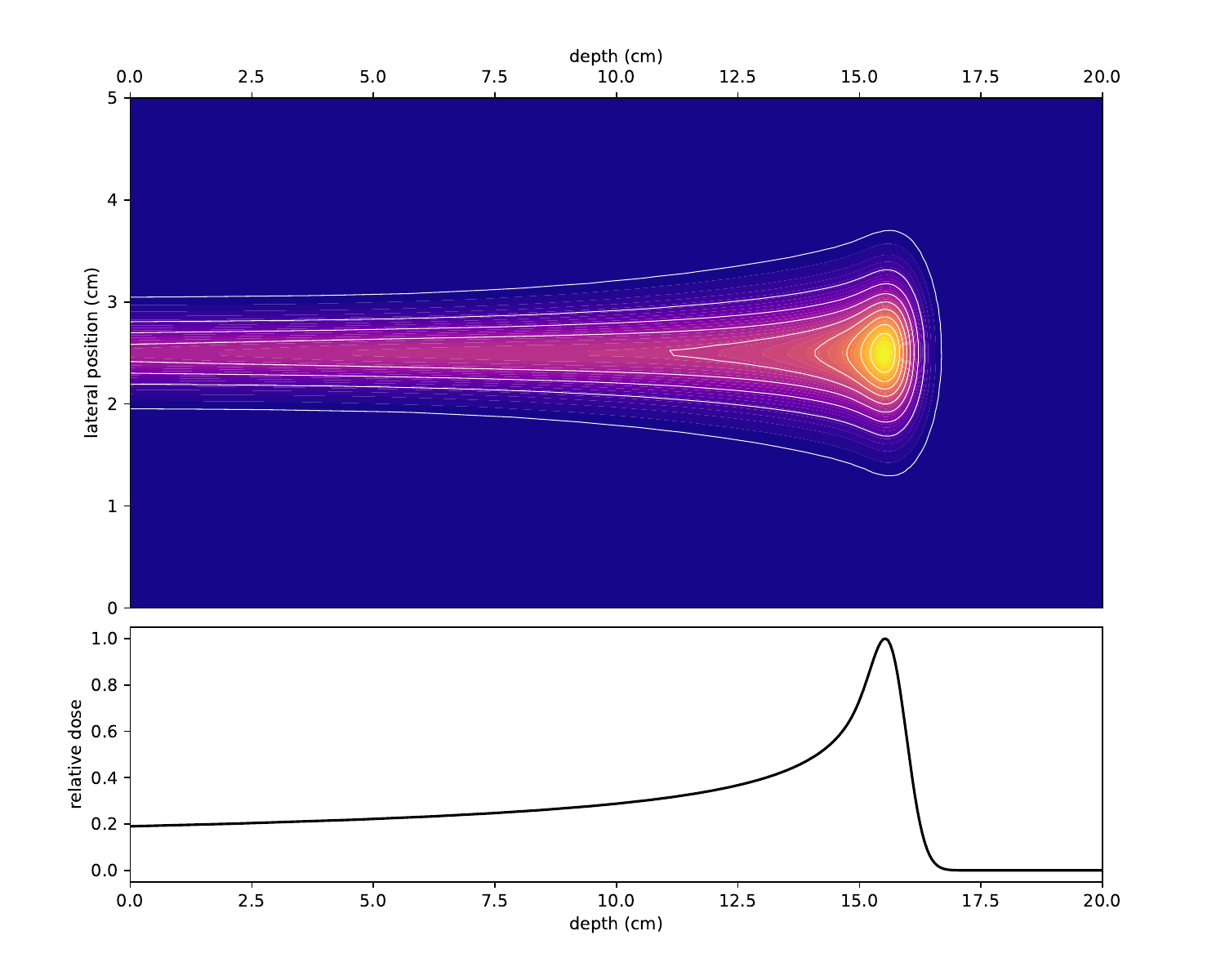}
  \captionof{figure}{
    \label{fig:BraggPeak}
    Simulated dose profile of a proton beam illustrating the Bragg
    peak. The simulation was performed using MCsquare with $1.21 \cdot
    10^7$ particles, a beam width of 2 mm, without nuclear
    interactions. The initial proton energy is 150 MeV with a 1\%
    energy spread, and the dose is integrated along the plane
    orthogonal to the beam axis. Figure from
    \cite{ashby2025efficient}.}
\end{wrapfigure}

The principal advantage of protons lies in their characteristic
\emph{Bragg peak}, a sharp energy deposition near the end of their
range, which contrasts with the more diffuse dose profile of
photons. This property enables precise dose shaping and better sparing
of critical structures, making proton therapy particularly well suited
for paediatric cases, head and neck cancers, and tumours adjacent to
radiosensitive organs \cite{thomas2020paediatric}.

Figure~\ref{fig:BraggPeak} shows a simulated proton dose profile
illustrating the Bragg peak. The simulation, performed using MCsquare
with $1.21 \cdot 10^7$ particles and a narrow beam width of 2 mm,
captures the distinctive energy deposition behaviour of 150 MeV
protons in tissue-equivalent material \cite{souris2016fast}.

Despite these advantages, the physics of proton transport is
complex. Accurate modelling must account for energy loss due to
ionisation, multiple Coulomb scattering, and non-elastic nuclear
reactions. Each of these effects introduces uncertainty and
nonlocality into the transport process, making predictive modelling a
formidable challenge. Two principal approaches have emerged in
response: deterministic models based on integro-differential transport
equations, and stochastic models based on Monte Carlo simulation of
individual particle tracks.

Deterministic approaches offer computational efficiency and analytic
structure \cite{ashby2025efficient}. They are typically formulated
using the Boltzmann or Boltzmann--Fokker--Planck equation and serve as
the foundation for many currently used clinical treatment planning
systems \cite{lin2017benchmarking}. However, they often rely on approximate
closures or empirical fits to handle angular scattering and nuclear
interactions \cite{pomraning1992fokker}. In contrast, stochastic
models offer high physical fidelity by directly simulating microscopic
interactions \cite{haghighat2020monte,lux2018monte}. Monte Carlo
methods are widely used for benchmarking and secondary dose
calculations, but can easily become too expensive for routine
optimisation or real-time adaptation
\cite{unkelbach2018robust,fredriksson2012characterization}.

This paper addresses the need for a rigorous mathematical connection
between deterministic and stochastic frameworks for proton
transport. This connection enables the development of hybrid
computational strategies that combine the physical accuracy of Monte
Carlo with the analytic structure and optimisation capabilities of
deterministic PDEs \cite{roos2008robust}. The resulting framework
supports not only accurate dose calculation but also efficient and
provably correct gradient-based treatment planning
\cite{paganetti2025proton}.

\subsection{Main Contributions}

This work develops a unified mathematical theory of proton transport
that links deterministic and stochastic formulations through
functional analysis, operator theory and stochastic calculus. Starting
from a physical description of energy loss and scattering, we
formulate the Boltzmann--Fokker--Planck equation in a variational
framework. We define the transport operator in a kinetic graph space
and establish monotonicity, maximality and the existence of well-posed
forward and adjoint problems. These results provide the foundation for
the description of deterministic PDE-based transport.

In parallel, we introduce a stochastic model of proton motion defined
via a system of stochastic differential equations (SDEs) indexed by
\emph{track length} rather than time. This model incorporates
continuous energy loss, angular diffusion on the sphere, and discrete
scattering events. We derive the infinitesimal generator of the
stochastic process and show that it corresponds to the backward
Boltzmann--Fokker--Planck operator.

Using semigroup theory and variational analysis, we prove that the
deterministic operator is the adjoint of the stochastic generator and
that the resolvent of the stochastic process coincides with the
Green’s function of the deterministic transport equation. This result
provides a rigorous mathematical basis for identifying the
\emph{fluence} in the deterministic setting with the \emph{expected
occupation density} in the stochastic setting.

We then apply this framework to the problem of dose computation. We
show that the classical expression
\begin{equation*}
  \text{dose} = \text{fluence} \times \text{mass stopping power}
\end{equation*}
emerges naturally and consistently from both deterministic and
stochastic perspectives. This unification allows us to express dose as
either a linear functional of the solution to the transport equation
or as the expected accumulated energy loss of a stochastic particle.

Finally, we apply these ideas to formulate a hybrid optimisation
strategy for proton therapy treatment planning. The optimisation
problem is posed over the beam intensity function on the inflow
boundary. The forward dose is computed using Monte Carlo simulations
of the SDE model, while gradients are computed efficiently via
deterministic adjoint PDEs. We prove existence of optimal controls
under realistic constraints, derive the first-order optimality
conditions in the form of a projected primal-dual system and propose a
numerical algorithm based on alternating stochastic simulation and
adjoint-based gradient updates.

\subsection{Relation to the Literature}

The mathematical modelling of particle transport has been the subject
of extensive work in both physics and applied
mathematics. Deterministic approaches include discrete ordinates
methods \cite{calloo2025cycle}, Fokker--Planck approximations and
moment-based closures, with clinical implementations often relying on
simplifying assumptions to reduce computational cost
\cite{wagner2011review,roos2008robust}. Stochastic methods, primarily
Monte Carlo, provide detailed physics and are used extensively for
verification and secondary dose estimates
\cite{haghighat2020monte,lux2018monte}. However, the computational
burden of Monte Carlo simulations poses challenges for real-time
optimisation settings
\cite{unkelbach2018robust,fredriksson2012characterization}.

Several works have attempted to introduce adjoint methods into Monte
Carlo frameworks, using techniques such as track reversal or
score-function estimation
\cite{difilippo1996adjoint,lux2018monte}. These approaches allow for
sensitivity analysis but typically require specialised sampling
techniques and are sensitive to variance. In contrast, our approach
leverages the duality between stochastic generators and deterministic
PDEs to obtain gradients via deterministic adjoint equations, while
retaining the physical accuracy of Monte Carlo for forward dose
computation. Moving forward, we believe this hybrid formulation will
enable efficient and robust real-time treatment planning strategies
\cite{paganetti2025proton,georgiou2025scotty}.

The mathematical connection between deterministic and stochastic
transport is well known in kinetic theory but has not been fully
developed in the context of proton therapy. Our work builds on ideas
from operator theory, semigroup analysis and stochastic calculus to
establish this connection rigorously, with direct implications for
computational radiotherapy
\cite{oksendal2013stochastic,pavliotis2014stochastic,lions1971optimal}.

The remainder of the paper is structured as
follows. Section~\ref{sec:physics} introduces the physical processes
governing proton transport, including energy loss and
scattering. Section~\ref{sec:deterministic} formulates the
deterministic transport equation in a variational framework and
establishes the properties of the forward and adjoint transport
operators. Section~\ref{sec:stochastic} develops the stochastic SDE
model and derives its generator and associated Kolmogorov
equations. Section~\ref{sec:resolvent} shows the duality between
stochastic and deterministic operators and constructing the
corresponding resolvents provides a unified theory of dose computation
in both settings. Section~\ref{sec:treatment} formulates the treatment
planning problem as a PDE-constrained stochastic optimisation problem
and proposes a hybrid algorithm. Section~\ref{sec:conclusion}
summarises the results and discusses future directions.

\section{Physics of Proton Transport}
\label{sec:physics}

To mathematically describe the evolution of protons, we work in the
spatial-direction-energy phase space. Let $\vec{x} \in D \subset
\mathbb{R}^3$ denote a position in a closed, bounded spatial domain,
$\omega \in \mathbb{S}^2$ a unit vector describing the direction of
motion, and $E \in \mathcal{E} := (E_{\min}, E_{\max}) \subset (0,
\infty)$ an admissible energy. The phase space is then given by
\begin{equation}
  \mathcal{C} = D \times \mathbb{S}^2 \times \mathcal{E}.
\end{equation}
Elements of $\mathcal{C}$ will typically be written as $X = (\vec{x},
\omega, E)$; that is, capital letters denote points in $\mathcal{C}$,
while $\vec{x} \in D$, $\omega \in \mathbb{S}^2$, and $E \in
\mathcal{E}$. Variants of these dummy variables will be used
throughout, but the notation will consistently distinguish spaces by
font and capitalisation. We also denote by $\vec{n}$ the
outward-pointing unit normal on the spatial boundary $\partial D$.

Protons interact with the medium primarily through three mechanisms:
inelastic collisions with atomic electrons, elastic Coulomb scattering
with nuclei, and non-elastic nuclear reactions (see
Figure~\ref{fig:atom}). These interactions are characterised by cross
sections, which describe the probability of an interaction occurring
per unit length travelled. They form the basis for scattering kernels,
stopping power functions, and reaction terms in the governing
transport equations.
\begin{figure}[h!]
    \centering
    \includegraphics[width=0.5\textwidth]{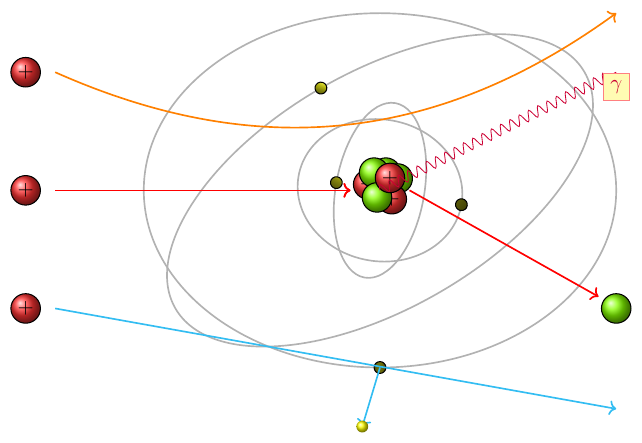}
    \caption{\em
      The three main interactions of a proton with matter. A
    \textcolor{red}{nonelastic} proton-nucleus collision, an
    \textcolor{cyan!80!white}{inelastic} Coulomb interaction with atomic
    electrons and \textcolor{orange}{elastic} Coulomb scattering with
    the nucleus.
    \label{fig:atom}
    }
\end{figure}

\subsection{Transport Processes in Proton Therapy}
\label{subsec:transport-processes}

A proton traversing a medium undergoes multiple interactions that
influence its trajectory, energy and ultimately its dose deposition
profile. In the absence of interactions, a proton moves in a straight
line with constant kinetic energy $E$. This free-streaming motion is
referred to as \emph{transport}.

As a proton propagates through an atomic medium, it experiences
collisional interactions with orbital electrons, transferring energy
while largely maintaining its direction. This process, known as
\emph{inelastic electron--proton interaction}, is the primary
mechanism of continuous energy loss and can lead to ionisation or
excitation of electrons. The frequency of such events is characterised
by the cross section $\sigma_i(\vec{x}, E' \to E)$, where $E'$ is the
proton energy prior to the interaction and $E' - E$ is the energy
lost.

Protons are also subject to small-angle deflections resulting from
elastic Coulomb interactions with nuclei. This \emph{elastic
nucleus--proton interaction} preserves total kinetic energy but
redistributes it between the proton and the nucleus, producing a net
energy loss over many such events. The angular deflection is described
by the cross section $\sigma_e(\vec{x}, \omega' \cdot \omega, E')$,
where $\omega' \cdot \omega$ is the cosine of the scattering angle and
$E'$ is the pre-collision energy.

At higher energies, a proton may undergo a non-elastic nuclear
reaction, producing secondary particles and incurring a discrete loss
of energy. This \emph{non-elastic proton--nucleus interaction} occurs
with probability per unit path length given by the cross section
$\sigma_{\mathrm{n}}(\vec{x}, \omega, E)$. Conditional on such an event, the
proton transitions from an incoming state $(\vec{x}, \omega', E')$ to
an outgoing state $(\vec{x}, \omega, E)$ with probability measure
$\pi(\vec{x}, \omega', E'; \dd\omega, \dd E)$.

\subsection{Energy, Range and Stopping Power}
\label{sec:stopping}

The range-energy relationship provides an estimate of how far protons
of a given energy can penetrate into a homogeneous medium. Suppose a
proton beam consists of particles with initial energy $E_0$. The
corresponding range $R_0$ is often modelled by a power law
\begin{equation}
  R_0 = \alpha E_0^p,
  \label{eq:range_energy}
\end{equation}
where $p \in [1,2]$ and $\alpha > 0$ are empirical constants
determined by the mass density and composition of the
medium. Indicative values for different biological tissues are given
in Table~\ref{tab:range_energy}.

\begin{table}[h!]
  \centering
  \begin{tabular}{lrr}
    \toprule
    Medium & \( p \) & \( \alpha \) \\
    \midrule
    Water & 1.75$\pm$0.02 & 0.00246$\pm$0.00025 \\
    Muscle & 1.75 & 0.0021 \\
    Bone & 1.77 & 0.0011 \\
    Lung & 1.74 & 0.0033 \\
    \bottomrule
  \end{tabular}
  \caption{
    \label{tab:range_energy}
    Range-energy relationship parameters for different media. The
    exponent $p$ remains relatively constant across tissues, while
    $\alpha$ varies significantly with density and composition. The
    uncertainty for water reflects variability reported in
    \cite{pettersen2018accuracy, boon1998dosimetry,
      bortfeld1997analytical}.}
\end{table}

Let $z$ denote the depth along the proton beam axis. At a given depth
$z$, the remaining range of the proton beam is $R_0 - z$, and applying
\eqref{eq:range_energy} to this residual range yields
\begin{equation}
  E(z) = \alpha^{-\frac{1}{p}} \left( R_0 - z \right)^{\frac{1}{p}}.
  \label{eq:E_of_x}
\end{equation}
The \emph{stopping power}, defined as the energy loss per unit path
length, follows by differentiation
\begin{equation}
  S(z) := -\frac{\d E(z)}{\d z}
  =
  \frac{\alpha^{-\frac{1}{p}}}{p} \left( R_0 - z \right)^{\frac{1}{p} - 1}.
\end{equation}
Inverting \eqref{eq:E_of_x} to express $z$ in terms of $E$ and
substituting into $S(z)$ yields the \emph{Bragg--Kleeman rule}
\begin{equation}
  S(E) = \frac{1}{\alpha p} E^{1-p},
  \label{eq:BraggKleeman}
\end{equation}
which shows how stopping power decreases with increasing energy,
giving rise to the characteristic Bragg peak near the end of the
range.

\begin{remark}[Heterogeneous Tissue Effects]
  The range-energy relationship assumes a homogeneous medium, such as
  a water phantom, where $\alpha$ and $p$ are constant. In
  heterogeneous tissues, the mass density and elemental composition
  vary with position, leading to spatial variations in $\alpha$ and
  discontinuities in the stopping power across tissue interfaces. In
  such cases, the range-energy relationship must be defined piecewise,
  with parameters adjusted for each region. This framework underlies
  practical models that incorporate CT-derived maps of tissue
  composition \cite{cox2024bayesian}.
\end{remark}

\section{Deterministic Formulation: The PDE Approach}
\label{sec:deterministic}

\subsection{Functional Setup for Proton Transport}
\label{subsec:functional-setup}

The mathematical analysis of the proton transport equation requires a
precise functional analytic framework. We adopt standard Sobolev and
Lebesgue space notations, following \cite{Evans:1998}.

For a Hilbert space $\mathcal{X} = \mathcal{X}(\mathcal{C})$ over the
phase space $\mathcal{C} = D \times \mathbb{S}^2 \times \mathcal{E}$,
we write $\| \cdot \|_{\mathcal{X}}$ and
$\ltwop{\cdot}{\cdot}_{\mathcal{X}}$ for the norm and inner product,
respectively. We denote the dual of $\mathcal{X}$ by $\mathcal{X}^*$,
and write $\ltwop{\cdot}{\cdot}_{\mathcal{X}(\mathcal{C})}$ for the
duality pairing between $\mathcal{X}^*$ and $\mathcal{X}$ when the
context is clear. Where needed for clarity, we write
$\ltwop{\cdot}{\cdot}_{\mathcal{X}, \mathcal{X}^*}$.

To set up the transport equation with boundary conditions, we first
introduce function spaces adapted to the phase-space boundary. The
spatial boundary of $\mathcal{C}$ is given by
\begin{equation}
  \Gamma = \partial D \times \mathbb{S}^2 \times \mathcal{E},
\end{equation}
where $\vec{n}$ denotes the outward-pointing unit normal to $\partial
D$. The outgoing  and incoming parts of the physical-energy boundary
are defined  as
\begin{equation}
  \begin{split}
    \Gamma^+
    &:=
    \{ (\vec{x}, \omega, E) \in \Gamma \mid \omega \cdot \vec{n} \geq 0 \text{ or } E = E_{\min} \}, \\
    \Gamma^-
    &:=
    \{ (\vec{x}, \omega, E) \in \Gamma \mid \omega \cdot \vec{n} < 0 \text{ or } E = E_{\max} \},
  \end{split}
\end{equation}
corresponding to the outflow and inflow boundaries, respectively, where we recall $\vec{n}$ the
outward-pointing unit normal on the spatial boundary $\partial D$.

The spaces $\leb{2}(\Gamma^{+})$ and $\leb{2}(\Gamma^{-})$ represent
square-integrable boundary traces weighted by the normal component of
the velocity, and are defined by
\begin{equation}
  \leb{2}(\Gamma^{\pm}) :=
  \left\{
    \phi : \Gamma^{\pm} \to \mathbb{R}
    \,\middle|\,
    \int_{\Gamma^{\pm}} \phi^2 |\omega \cdot \vec{n}| \, \dd \gamma < \infty
  \right\},
\end{equation}
where $\dd \gamma$ denotes the product of the surface measure on
$\partial D$, the uniform surface measure on $\mathbb{S}^2$, and the
Lebesgue measure on $\mathcal{E}$. These are Hilbert spaces under the
weighted inner product
\begin{equation}
  \ltwop{\phi_1}{\phi_2}_{\Gamma^{\pm}} :=
  \int_{\Gamma^{\pm}} \phi_1 \phi_2 |\omega \cdot \vec{n}| \, \dd \gamma.
\end{equation}

\subsection{The Boltzmann--Fokker--Planck Operator}

The full Boltzmann--Fokker--Planck (BFP) operator governing proton
transport is denoted by $\fL$ and given by
\begin{equation}
  \fL \phi := \fT \phi - \mu \Delta_{\omega} \phi,
\end{equation}
where $\mu > 0$ is the angular diffusion coefficient and the transport
operator $\fT$ is defined by
\begin{equation}
  \fT \phi := \omega \cdot \nabla_{\vec{x}} \phi 
  - \partial_E (S \phi) 
  - \cS \phi + \sigma_{\mathrm{n}} \phi.
\end{equation}
Here, $S$ is the stopping power from \eqref{eq:BraggKleeman}, and
$\cS$ is the scattering operator describing angular and energy
redistribution due to collisions.

We define $\cS$ in terms of elastic and non-elastic components as
\begin{align}
  \cS \phi(\vec{x}, \omega, E) 
  &=
  \int_{\mathbb{S}^2}\int_{\mathcal{E}} \phi(\vec{x}, \omega', E') \, \sigma_{\mathrm{n}}(\vec{x}, \omega', E') \,
  \pi(\vec{x}, \omega', E';  \omega,  E)\dd \omega' \dd E',
  \label{eq:bigS_def}
\end{align}
where the total scattering cross section $\sigma_{\mathrm{n}}$ is decomposed into
elastic and non-elastic parts
\begin{equation}
  \label{eq:bigS}
  \sigma_{\mathrm{n}} = \sigma_{\mathrm{e}} + \sigma_{\mathrm{ne}}.
\end{equation}
The  kernel $\pi$ defining post-collision transitions is then
\begin{equation}
  \pi(\vec{x}, \omega', E';  \omega,  E)
  =
  \frac{\sigma_{\mathrm{e}}(\vec{x}, \omega', E')}{\sigma_{\mathrm{n}}(\vec{x}, \omega', E')}
  \pi_{\mathrm{e}}(\vec{x}, \omega', E';  \omega) \, \delta_{E'}( E)
  +
  \frac{\sigma_{\mathrm{ne}}(\vec{x}, \omega', E')}{\sigma_{\mathrm{n}}(\vec{x}, \omega', E')}
  \pi_{\mathrm{ne}}(\vec{x}, \omega', E';  \omega,  E),
  \label{eq:pi_mixture_kernel}
\end{equation}
where $\delta_{E'}$ denotes the Dirac density at energy $E'$. 
In this decomposition, $\sigma_{\mathrm{e}}(\vec{x}, \omega', E')$ and
$\pi_{\mathrm{e}}(\vec{x}, \omega', E'; \omega)$ represent the elastic
cross section and angular kernel from $(\vec{x}, \omega', E')$ to
$(\vec{x}, \omega, E)$, while $\sigma_{\mathrm{ne}}(\vec{x}, \omega',
E')$ and $\pi_{\mathrm{ne}}(\vec{x}, \omega', E'; \omega, E)$ model
non-elastic scattering events with energy redistribution.
The 
mixture kernel \eqref{eq:pi_mixture_kernel} defines a probability distribution over post-scattering
states $(\omega, E)$, conditioned on the incoming state $(\vec{x},
\omega', E')$.

In particular, we require that
\begin{equation}
  \int_{\mathbb{S}^2} \int_{\mathcal{E}} 
  \pi(\vec{x}, \omega', E';  \omega,  E)\dd \omega\dd E = 1,
  \qquad \text{for all } (\vec{x}, \omega', E') \in \mathcal{C},
  \label{eq:pi_is_prob_dist}
\end{equation}
ensuring that $\cS$ is a positive contraction.

\subsection{Functional Spaces and Operator Domains}

We define the kinetic graph space
\begin{equation}
  \cH
  :=
  \left\{
  \phi \in \leb2(\mathcal{C})
  ~\middle|~
  \nabla_\omega \phi \in \leb2(\mathcal{C})^3, \quad
  \fT \phi \in \leb2(\mathcal{C})
  \right\},
\end{equation}
equipped with the graph norm
\begin{equation}
  \Norm{\phi}_{\cH}^2
  :=
  \Norm{\phi}_{\leb2(\mathcal{C})}^2
  + \Norm{\sqrt{\mu} \nabla_\omega \phi}_{\leb2(\mathcal{C})}^2
  + \Norm{\fT \phi}_{\leb2(\mathcal{C})}^2.
\end{equation}
We define the subspace with homogeneous inflow boundary conditions as
\begin{equation}
  \cH_0^-
  :=
  \left\{
  \phi \in \cH ~\middle|~
  \phi = 0 \text{ on } \Gamma^-
  \right\},
\end{equation}
corresponding to physical inflow constraints for the forward transport
problem.

To define the range space of the BFP operator $\fL$, we introduce
\begin{equation}
  \cV := \leb2(D \times \mathcal{E}; H^{-1}(\mathbb{S}^2)),
\end{equation}
which encodes $\leb2$ regularity in position and energy, and dual
Sobolev regularity in angle. This space is appropriate for the weak
formulation of the degenerate elliptic transport equation.

Note that for its dual $\cV^*$, we have continuous embeddings $\cH,
\cH_0^- \subset \cV^*$. We write $\ltwop{\cdot}{\cdot}$ to denote the
duality pairing between $\cV$ and $\cV^*$, omitting explicit mention
of the spaces when context permits.

The natural boundary conditions for proton transport arise from
specifying the incoming flux on the inflow boundary $\Gamma^-$. For
solutions in $\cH$, we consider the boundary value problem
\begin{equation}
  \label{eq:BTE}
  \fL \psi = 0 \quad \text{in } \mathcal{C}, \qquad \psi = g \quad \text{on } \Gamma^-,
\end{equation}
where $g$ is a prescribed inflow intensity. In the stochastic
formulation, this corresponds to the \emph{entry distribution} of
particles, specifying the initial sampling law. At high energies, we
assume decay of the solution as $E \to E_{\max}$, which reflects
physical attenuation.

\subsection{The Backward Boltzmann--Fokker--Planck Operator}

We define the functional analytic setting for the backward
Boltzmann--Fokker--Planck (bBFP) operator. The domain of the backward
transport operator is
\begin{equation}
    \cbH 
    := 
    \left\{
    \phi \in \leb2(\mathcal{C})
    ~\middle|~
    \bT \phi \in \leb2(\mathcal{C}),\ 
    \nabla_{\omega} \phi \in \leb2(\mathcal{C})^3
    \right\},
\end{equation}
with graph norm
\begin{equation}
    \Norm{\phi}_{\cbH}^2
    :=
    \Norm{\phi}^2_{\leb2(\mathcal{C})} 
    + \Norm{\sqrt{\mu} \nabla_{\omega} \phi}^2_{\leb2(\mathcal{C})}
    + \Norm{\bT \phi}^2_{\leb2(\mathcal{C})}.
\end{equation}

We also define the subspace of functions vanishing on the outflow
(i.e., terminal) boundary,
\begin{equation}
    \cbH_0^+ :=
    \left\{
    \phi \in \cbH ~\middle|~
    \phi = 0 \text{ on } \Gamma^+
    \right\}.
\end{equation}
This space enforces terminal conditions on the physical and energy
outflow boundaries, corresponding to the vanishing of expected future
contributions in the adjoint formulation.

The bBFP operator is defined formally as the dual of the forward
operator $\fL$ with respect to the $\leb2(\mathcal{C})$ inner product
\begin{equation}
  \ltwop{\fL \phi}{\psi} = \ltwop{\phi}{\bL \psi}
  \quad
  \text{for all } \phi \in \cH_0^-,~ \psi \in \cbH_0^+.
  \label{dualinnerproduct}
\end{equation}

\begin{theorem}[Adjoint representation of the Boltzmann–Fokker–Planck operator]
  It follows from the definition \eqref{dualinnerproduct}
  that
  \begin{equation}\label{backwards_BFP} \bL \phi := \bT \phi -
    \mu \Delta_{\omega} \phi,
  \end{equation}
  where the adjoint transport operator $\bT$ is
  \begin{equation}
    \bT \phi := - \omega \cdot \nabla_{\vec x} \phi 
    + S \partial_E \phi 
    - \cS^* \phi + \sigma_{\mathrm{n}} \phi;
  \end{equation}
  and the adjoint scattering operator $\cS^*$ is given by
  \begin{equation}
    \cS^* \psi(\vec x,\omega,E) := \sigma_{\mathrm{n}} (\vec x,\omega,E)\int_{\mathbb{S}^2} \int_{\mathcal{E}}
    \psi(\vec x, \omega', E') \pi(\vec x,\omega, E;  \omega',   E')\dd \omega'  \dd E',
  \end{equation}
  corresponding to a reversal of pre- and post-collision
  roles in the scattering process.
\end{theorem}
\begin{proof}
  The duality relation \eqref{dualinnerproduct} follows from formal
  integration by parts, under the assumption that boundary terms
  vanish due to the support properties of $\phi$ and $\psi$. Similar
  computations are carried out rigorously in, e.g., \cite{DL6}.
\end{proof}

%

\subsection{Admissible Inflow Boundary Conditions}

The space of admissible inflow boundary data consists of functions in
$\leb2(\Gamma^-)$ that arise as traces of elements in $\cH$, the
forward graph space. We define the trace space as
\begin{equation}
    {\rm{tr}}(\Gamma^-) := \left\{
    g \in \leb2(\Gamma^-)
    ~\middle|~
    \exists G \in \cH \text{ such that } \gamma^- (G) = g
    \right\},
\end{equation}
where $\gamma^-$ denotes the trace operator on $\Gamma^-$. This space
ensures compatibility between the boundary data and the variational
formulation of the forward problem. In the stochastic formulation,
${\rm tr}(\Gamma^-)$ corresponds to the set of admissible entry
distributions for particles entering the phase space.

\subsection{Reformulation as an Interior Problem}

The natural variational space for the forward problem is $\cH_0^-$,
comprising functions that vanish on the inflow boundary
$\Gamma^-$. However, the boundary value problem
\begin{equation}
  \label{eq:BFPE}
  \fL \psi = 0 \quad \text{in } \mathcal{C}, \qquad \psi = g \quad \text{on } \Gamma^-
\end{equation}
prescribes inhomogeneous boundary data $g$ on $\Gamma^-$. To express
this within the variational framework, we introduce a \emph{lifting
function} $\widehat g \in \cH$ such that $\gamma^- (\widehat g) = g$. Define the
shifted unknown
\begin{equation}\label{eq:extension_of_g}
  u := \psi - \widehat g,
\end{equation}
so that $u \in \cH_0^-$ satisfies homogeneous inflow boundary
conditions.

Substituting into the equation, we find that $u$ satisfies the
inhomogeneous problem
\begin{equation}\label{eq:extension_of_g_problem}
  \fL u = -\fL \widehat g
  :=
  f \quad \text{in } \mathcal{C}, \qquad u = 0 \quad \text{on } \Gamma^-.
\end{equation}
This allows us to recast the original boundary value problem as a
variational equation: find $u \in \cH_0^-$ such that
\begin{equation}
  \ltwop{\fL u}{\phi} = \ltwop{f}{\phi} = -\ltwop{\fL \widehat g}{\phi}, \qquad \forall \phi \in \cH_0^-.
  \label{eq:BFPE-lifted}
\end{equation}

This formulation is well defined provided that $\widehat g \in \cH$ and
$\fL \widehat g \in \cV^*$. The lifting $\widehat g$ can be constructed
by solving an auxiliary extension problem or interpolating $g$ into a
function in $\cH$, cf.~\cite[Section~5.4]{Evans:1998}.

In the probabilistic framework, this lifting corresponds to
decomposing the solution $\psi$ into two parts: a component $\widehat
g$ that encodes the entry distribution of particles on $\Gamma^-$, and
a remainder $u$ that captures the contribution of those trajectories
that evolve within the domain after entering. Since the process is
absorbed at the outflow boundary $\Gamma^+$ or at the minimal energy
$E_{\min}$, the function $u$ describes the internal evolution of
particles conditioned on entering at {$\Gamma^-$}, the entrance
boundary, and vanishes on the {exit boundary $\Gamma^+$}.

To establish well-posedness of \eqref{eq:BFPE} (and hence of the
original boundary value problem), we show that the operator $\fL$ is
both \emph{monotone} and \emph{maximal}
\cite[§7.1]{brezis2011functional}. Under these conditions, existence
and uniqueness of a solution $u \in \cH_0^-$ follow, and $\cH_0^-$
inherits a Hilbert space structure with inner product defined
analogously to its graph norm \cite[Section~6.3]{Ern_Guermond_1}. As
we will see, in the stochastic setting, maximality of $\fL$
corresponds to well-posedness of the resolvent problem for $\bL$,
ensuring that the backward process is defined throughout the entire
phase space.

\begin{lemma}[Monotonicity of the BFP Operator]
  \label{lemma:monotone}
  Assume that the angular diffusion coefficient $\mu \in
  L^{\infty}(D \times \mathcal{E}; \mathbb{R}^+)$, and that the
  stopping power $S \in W^{1,\infty}(D \times \mathcal{E})$
  satisfies
  \begin{equation}
    \lim_{E \downarrow E_{\min}} S(\vec x, E) \geq 0,~ \forall \vec x \in D, 
    \qquad 
    \partial_E S(\vec x, E) \leq 0,~ \forall (\vec x, E) \in D \times \mathcal{E}.
  \end{equation}
  Additionally, suppose that the nuclear scattering cross-section satisfies
  $\sigma_{\mathrm{n}}(\vec{x}, \omega, E) \geq 0$, and that the scattering kernel
  $\pi$ satisfies the normalisation condition \eqref{eq:pi_is_prob_dist}.
  Then, there exist constants $C_{\rightarrow}, C_{\leftarrow} > 0$ such that the
  Boltzmann--Fokker--Planck operator $\fL$ and its adjoint $\bL$ are
  monotone in the sense that
  \begin{equation}
    \begin{aligned}
        \ltwop{\fL u}{u} 
        &\geq
        C_{\rightarrow} \left(\|u\|^2_{\Gamma^+} 
        + \|\sqrt{\mu}\nabla_{\omega}u\|^2_{\leb2(\mathcal{C})} 
        + \|u\|^2_{\leb2(\mathcal{C})} 
        + \lim_{E \downarrow E_{\min}} \|\sqrt{S}u\|^2_{\leb2(D\times \mathbb{S}^2)} \right),
        \quad \forall u \in \cH_0^-, 
    \\ 
        \ltwop{\bL \phi}{\phi} 
        &\geq 
        C_{\leftarrow} \left(\|\phi\|^2_{\Gamma^-} 
        + \|\sqrt{\mu}\nabla_{\omega}\phi\|^2_{\leb2(\mathcal{C})} 
        + \|\phi\|^2_{\leb2(\mathcal{C})} 
        + \lim_{E \uparrow E_{\max}} \|\sqrt{S}\phi\|^2_{\leb2(D\times \mathbb{S}^2)} \right), 
        \quad \forall \phi \in \cbH_0^+.
    \end{aligned}
  \end{equation}
\end{lemma}

\begin{proof}
We first consider the transport term. Integration by parts yields
\begin{equation}
  \ltwop{\omega \cdot \nabla_{\vec x} u}{u}
  =
  \frac{1}{2} \|u\|^2_{\Gamma^+}
  - \frac{1}{2} \|u\|^2_{\Gamma^-}.
\end{equation}
Since $u \in \cH_0^-$ vanishes on $\Gamma^-$, it follows that
\begin{equation}
  \ltwop{\omega \cdot \nabla_{\vec x} u}{u}
  = \frac{1}{2} \|u\|^2_{\Gamma^+} \geq 0.
\end{equation}
Similarly, for the adjoint operator,
\begin{equation}
  -\ltwop{\omega \cdot \nabla_{\vec x} \phi}{\phi}
  = \frac{1}{2} \|\phi\|^2_{\Gamma^-} \geq 0,
  \quad \forall \phi \in \cbH_0^+.
\end{equation}
For the angular diffusion term
\begin{equation}
  -\ltwop{\mu \Delta_\omega v}{v}
  = \|\sqrt{\mu} \nabla_\omega v\|^2_{\leb2(\mathcal{C})} \geq 0,
  \quad \forall v \in \cH_0^- \cup \cbH_0^+.
\end{equation}
We next consider the forward energy term
\begin{align}
  -\ltwop{\partial_E(S u)}{u}
  &= - \ltwop{S \partial_E u}{u} - \ltwop{(\partial_E S) u}{u} \\
  &= \frac{1}{2} \lim_{E \downarrow E_{\min}} \|S^{1/2} u\|^2_{L^2(D \times \mathbb{S}^2)}
  - \frac{1}{2} \ltwop{(\partial_E S) u}{u}.
\end{align}
By the assumption $\partial_E S \leq 0$ and $S \geq 0$ near
$E_{\min}$, both terms are nonnegative. For the backward energy term
\begin{equation}
  \ltwop{S \partial_E \phi}{\phi}
  = - \frac{1}{2} \ltwop{\partial_E S \cdot \phi}{\phi} \geq 0.
\end{equation}
Finally, for the scattering terms
\begin{equation}
  \ltwop{(\sigma_{\mathrm{n}} - \cS) v}{v}
  = \int_{\mathcal{C}} \left( \sigma_{\mathrm{n}} |v|^2 - \cS v \cdot v \right) \geq 0,
\end{equation}
by the positivity and contraction property of $\cS$ (due to
normalisation of $\pi$) and similarly for $\bL$ and $\cS^*$. This
completes the proof.
\end{proof}

We now discuss the validity of the assumptions made in
Lemma~\ref{lemma:monotone}. The assumption that the angular diffusion
coefficient $\mu$ is nonnegative follows directly from the physical
constraint that the elastic scattering cross-section $\sigma_e$ is
nonnegative. Since $\mu$ is derived from angular moments of
$\sigma_e$, this condition is physically justified and ensures
well-posed angular diffusion.

The assumption $S \in W^{1,\infty}(D \times \mathcal{E})$ reflects the
use of regularised stopping power functions. For example, the
Bragg-Kleeman model expresses stopping power as
\begin{equation}
  S(\vec x, E) = \dfrac{1}{\alpha(\vec x)\, p(\vec x)} E^{1 - p(\vec x)},
\end{equation}
with $\alpha(\vec x) > 0$ and $p(\vec x) > 1$. This model satisfies
the monotonicity and boundedness conditions provided the admissible
energy domain satisfies $E_{\min} > 0$. In the unregularised case
$E_{\min} = 0$, the stopping power diverges as $E \to 0$, violating
the boundedness and integrability required for
Lemma~\ref{lemma:monotone}.

In practice, such singular models are not implemented directly due to
their nonphysical behaviour at low energies. Empirical stopping power
laws (e.g.~Bragg--Kleeman or Bethe--Bloch) are valid in the
continuous slowing-down regime at high energies, but break down near
rest energy. To address this, \emph{screened stopping power models}
are employed, where $S(\vec{x}, E)$ is regularised to remain bounded as
$E \to 0$. These corrections incorporate low-energy screening and
density effects. Further discussion can be found in
\cite{chronholm2025geometry}.

From the functional analytic perspective, the backward operator
$\bL : \cbH_0^+ \to \cV$ is closed, unbounded, and maximal monotone.
It admits a well-defined resolvent operator. For any $\lambda > 0$, the
resolvent
\begin{equation}
  \bR_\lambda := (\lambda I + \bL)^{-1} : \cV \to \cbH_0^+
\end{equation}
defines the unique solution $u \in \cbH_0^+$ to the variational problem
\begin{equation}
  \label{flowfrom}
  \lambda u + \bL u = v, \qquad \text{for } v \in \cV.
\end{equation}

In many classical settings, the case $\lambda = 0$ is excluded unless
additional coercivity conditions are imposed to ensure invertibility
of $\bL$.  However, under the assumptions of Lemma~\ref{lemma:monotone},
the operator $\bL$ contains dissipative terms, notably angular
diffusion and energy drift, that induce strict energy decay and
prevent nullspace degeneracy.  As a result, the resolvent at zero
\begin{equation}
  \bR_0 := \bL^{-1}
\end{equation}
is well-defined and corresponds to the unique variational solution of
the stationary problem $\bL u = v$ in $\cbH_0^+$.

This property reflects the fundamentally dissipative, non-conservative
nature of proton transport, where particles are eventually absorbed
either spatially or energetically. In the stochastic formulation, this
corresponds to finite lifetimes of sample paths, i.e., almost sure
termination of trajectories after a finite track length.

\begin{lemma}[Maximality of the BFP Operator]\label{lemma:closure_and_ker_0}
  Let $S, \mu, \sigma_{\mathrm{n}}$ and $\pi$ satisfy the assumptions of
  Lemma \ref{lemma:monotone}. Then the operators $\fL : \cH_0^- \to
  \cV$ anxd $\bL : \cbH_0^+ \to \cV$ are maximal monotone. In particular,
  \begin{equation}
    \operatorname{range}(I + \fL) = \operatorname{range}(I + \bL) = \cV.
  \end{equation}
\end{lemma}

\begin{proof}
We begin by showing that $\operatorname{range}(I + \fL)$ is closed in
$\cV$. Let $\{f_i\} \subset \operatorname{range}(I + \fL)$ with $f_i
\to f$ strongly in $\cV$. By definition, for each $i$, there exists
$u_i \in \cH_0^-$ such that
\begin{equation}
  (I + \fL) u_i = f_i.
\end{equation}
Taking the duality pairing with $u_i$, we obtain
\begin{equation}
  \ltwop{(I + \fL) u_i}{u_i} = \ltwop{f_i}{u_i}.
\end{equation}
By monotonicity from Lemma~\ref{lemma:monotone}, the left-hand side is
bounded below by the graph norm of $u_i$
\begin{equation}
  \ltwop{(I + \fL) u_i}{u_i} \geq C \|u_i\|_{\cH}^2,
\end{equation}
for some constant $C > 0$. Since $\ltwop{f_i}{u_i} \leq \|f_i\|_{\cV}
\|u_i\|_{\cH}$, we deduce that $\{u_i\}$ is bounded in $\cH_0^-$.

Hence, up to a subsequence, $u_i \rightharpoonup u$ weakly in
$\cH_0^-$. Since $\fL$ is closed (as a sum of closed operators), the
limit $u$ satisfies $(I + \fL) u = f$, showing that $f \in
\operatorname{range}(I + \fL)$. Thus, the range is closed in
$\cV$. The same argument applies to the adjoint operator $\bL :
\cbH_0^+ \to \cV$, showing that $\operatorname{range}(I + \bL)$ is
closed as well.

To establish maximality, we invoke duality of monotone operators (see
\cite[Thm.~7.1]{brezis2011functional}), which implies:
\begin{equation}
  \left( \operatorname{ker}(I + \fL) \right)^\perp = \operatorname{range}(I + \bL), 
  \qquad 
  \left( \operatorname{ker}(I + \bL) \right)^\perp = \operatorname{range}(I + \fL).
\end{equation}
It remains to show that $\operatorname{ker}(I + \fL) = \{0\}$ and similarly for $\bL$.

Suppose $u_0 \in \operatorname{ker}(I + \fL)$, so that $(I + \fL) u_0 = 0$. Then
\begin{equation}
  \ltwop{(I + \fL) u_0}{u_0} = 0.
\end{equation}
Using the decomposition from Lemma~\ref{lemma:monotone}, we have
\begin{align}
  \ltwop{(I + \fL) u_0}{u_0}
  &= \|u_0\|^2_{\leb2(\mathcal{C})}
  + \|\sqrt{\mu} \nabla_\omega u_0\|^2_{\leb2(\mathcal{C})}
  + \frac{1}{2} \|u_0\|^2_{\Gamma^+} \notag\\
  &\quad + \frac{1}{2} \lim_{E \downarrow E_{\min}} \|\sqrt{S} u_0\|^2_{\leb2(D \times \mathbb{S}^2)}
  - \frac{1}{2} \ltwop{(\partial_E S) u_0}{u_0}
  + \ltwop{(\sigma_{\mathrm{n}} - \cS) u_0}{u_0}.
\end{align}
Each term is nonnegative under the assumptions of
Lemma~\ref{lemma:monotone}, and therefore all must vanish. In
particular,
\begin{equation}
  \|u_0\|_{\leb2(\mathcal{C})} = \|\nabla_\omega u_0\|_{\leb2(\mathcal{C})} = \|u_0\|_{\Gamma^+} = 0,
\end{equation}
and so $u_0 = 0$ in $\cH_0^-$.

A similar argument applies to $\bL$ using its adjoint structure and
monotonicity, yielding $\operatorname{ker}(I + \bL) =
\{0\}$. Therefore, both operators are maximal monotone.
\end{proof}

\begin{theorem}[Existence and Uniqueness of Solutions]
  \label{thm:forward_exist_unique}
  Let the assumptions of Lemma~\ref{lemma:monotone} hold. Then, for each
  $f \in \cV$, there exists a unique solution $u \in \cH_0^-$
  satisfying
  \begin{equation}
    \ltwop{\fL u}{v} = \ltwop{f}{v}, \qquad \forall v \in \cbH_0^+.
  \end{equation}
  Similarly, for each $f \in \cV$, there exists a unique $z \in \cbH_0^+$ such that
  \begin{equation}
    \ltwop{\bL z}{v} = \ltwop{f}{v}, \qquad \forall v \in \cH_0^-.
  \end{equation}
\end{theorem}

\begin{proof}
  The result follows from the monotonicity of $\fL$ and $\bL$
  established in Lemma~\ref{lemma:monotone}, and the maximality proven
  in Lemma~\ref{lemma:closure_and_ker_0}. By the standard theory of
  maximal monotone operators in Hilbert spaces (see, e.g.,
  \cite[§6.3]{Ern_Guermond_1}), the variational problems admit unique
  solutions $u \in \cH_0^-$ and $z \in \cbH_0^+$.
\end{proof}

\section{Stochastic Formulation: The SDE Approach}
\label{sec:stochastic}

We now reformulate the proton transport problem from a probabilistic
perspective. Whereas the deterministic formulation models the
evolution of particle density via an integro-differential operator,
the stochastic formulation describes a Markovian trajectory in the
phase space $\mathcal{C}$. This viewpoint models the physical path
along which energy is deposited by a single initial proton, accounting
for its complete stochastic life cycle, including continuous energy
loss (Coulomb interaction), elastic scattering and the possible
generation of secondary or tertiary protons at non-elastic nuclear
interactions. These secondary particles further transport residual
energy (cf. \cite{smith2025stochastic}).

This trajectory-based representation naturally expresses dose as an
accumulated quantity along individual tracks and aligns with Monte
Carlo simulation techniques widely used for dose computation and
verification in clinical proton therapy. Mathematically, the
stochastic framework yields a dual characterisation of the
deterministic model, the infinitesimal generator of the process
corresponds to the adjoint transport operator and the associated
resolvent defines a probabilistic representation of fluence.

This section introduces the stochastic process and its governing SDEs,
and establishes its correspondence with the deterministic theory
developed earlier.

Unlike conventional SDEs, which are indexed by time, we follow the
convention of the nuclear physics literature and parameterise the
process by \emph{track length}.

We define the stochastic process
\begin{equation}
  Y = (Y_\ell,~ \ell \ge 0) = ((\vec{x}_\ell, \omega_\ell, E_\ell),~ \ell \ge 0),
\end{equation}
where $\vec{x}_\ell \in D$, $\omega_\ell \in \mathbb{S}^2$ and $E_\ell
\in \mathcal{E}$ denote the proton's position, direction and energy at
track length $\ell$. The process $Y$ describes the trajectory of a
single proton evolving in the phase space $\mathcal{C} = D \times
\mathbb{S}^2 \times \mathcal{E}$.

The evolution continues until a random stopping length $\Lambda$ defined by
\begin{equation}
  \Lambda := \inf\{\ell > 0 : E_\ell = E_{\min} \text{ or } \vec{x}_\ell \notin D \},
\end{equation}
that is, until the proton either exhausts its energy or leaves the
spatial domain. The stopping time $\Lambda$ depends on the initial
condition $(\vec{x}_0, \omega_0, E_0)$, though this dependence is
suppressed in notation.

To ensure the process is defined for all $\ell \geq 0$, we extend the
state space by introducing a cemetery state $\dagger$, and define
\begin{equation}
  Y_\ell := 
  \begin{cases}
    (\vec{x}_\ell, \omega_\ell, E_\ell), & \ell < \Lambda, \\
    \dagger, & \ell \geq \Lambda.
  \end{cases}
\end{equation}
The state $\dagger$ is absorbing, once entered, the process remains
there for all subsequent track lengths. To maintain consistency, we
adopt the convention that any test function $f : \mathcal{C} \cup
\{\dagger\} \to \mathbb{R}$ satisfies $f(\dagger) = 0$. This ensures
that $\dagger$ contributes nothing to expectations, integrals or
occupation measures, allowing us to work on the extended domain.

\subsection{Stochastic Model for Proton Transport}
\label{subsec:sde-model}

We now provide the differential form of the stochastic evolution of
the proton configuration $Y_\ell$, defined on the event $\{\ell <
\Lambda\}$. The dynamics are governed by a system of stochastic
differential equations on $\mathcal{C}$, incorporating both continuous
and jump-driven components.

Energy loss is driven by a continuous stopping power term and discrete
jumps due to nuclear scattering. Spatial motion proceeds along the current
direction $\omega_\ell$, while angular variation arises from both
Brownian diffusion and jump-induced deflections. Angular diffusion is
modelled by Brownian motion $B_\ell \in \mathbb{R}^3$, projected onto
the tangent space of $\mathbb{S}^2$ via the cross product
$\omega_{\ell-} \wedge \d B_\ell$ and correction term
$- \omega_{\ell-} \, \d \ell$, ensuring that
$\omega_\ell \in \mathbb{S}^2$ is maintained.

Discrete scattering events are described by a random jump measure
$N(Y_{\ell-}; \d\ell, \d\omega', \d E')$, with predictable compensator
\begin{equation}
\tilde{N}(Y; \d\ell, \d\omega', \d E') 
= \sigma_{\mathrm{n}}(Y)\pi(Y; \omega',  E') \d \omega' \d E'  \dd \ell,
\qquad 
\forall Y\in\mathcal{C},~ \omega'\in\mathbb{S}^2,~ E'\in \mathcal{E},
\end{equation}
where the scattering components $\sigma_{\mathrm{n}}$, $\sigma_{\rm e}$,
$\sigma_{\rm ne}$ and the kernel $\pi$ were introduced in
\eqref{eq:bigS}.

The resulting system of stochastic differential equations governing
$Y_\ell = (\vec{x}_\ell, \omega_\ell, E_\ell)$ is given by
\begin{equation}
\label{eq:sde-tracklength}
\begin{split}
  \d E_\ell 
  &= 
  -S(Y_{\ell-}) \, \d \ell 
  - \int_{\mathcal{E}} \int_{\mathbb{S}^2} (E_{\ell-} - E') \, N(Y_{\ell-}; \d\ell, \d\omega', \d E'), \\
  \d \vec{x}_\ell 
  &= 
  \omega_{\ell-} \, \d \ell, \\
  \d \omega_\ell 
  &= 
  -\mu(Y_{\ell-})^2 \omega_{\ell-} \, \d \ell 
  + \mu(Y_{\ell-}) \omega_{\ell-} \wedge \d B_\ell 
  + \int_{\mathcal{E}} \int_{\mathbb{S}^2} (\omega' - \omega_{\ell-}) \, N(Y_{\ell-}; \d\ell, \d\omega', \d E').
\end{split}
\end{equation}
See \cite{protonbeam}.
\begin{theorem}[Variational Representation of the Infinitesimal Generator]\label{C0}
  Let $Y_\ell = (\vec{x}_\ell, \omega_\ell, E_\ell)$ be the stochastic
  process defined by the SDE \eqref{eq:sde-tracklength}, extended with
  a cemetery state $\dagger$ after the stopping time $\Lambda$, as
  described above. Let $\bL: \cbH \to \cV$ denote the backward
  Boltzmann--Fokker--Planck operator defined in
  \eqref{backwards_BFP}. Then, for any $v \in \cbH_0^+$, the function
  \begin{equation}
    u(\ell,X) := \mathbb{E}_{X} [v(Y_\ell)\mathbf{1}_{\ell < \Lambda}], \qquad \forall X = (\vec{x}, \omega, E) \in \mathcal{C},
    \label{semigroupu}
  \end{equation}
  satisfies the Kolmogorov backward equation
  \begin{equation}
    \partial_\ell u + \bL u = 0 \quad \text{in } \cV, \qquad u(0) = v.
  \end{equation}
  Moreover, $u \in C([0, \infty); \cbH) \cap C^1((0, \infty); \cV)$,
    and the operator $\bL$ generates a strongly continuous contraction
    semigroup on $\cbH_0^+$.
\end{theorem}

\begin{proof}
  Define the operator family
  \begin{equation}
    T_\ell v(X) := \mathbb{E}_{X}[v(Y_\ell)\mathbf{1}_{\ell < \Lambda}], \qquad \forall v \in \cbH_0^+,~ X \in \mathcal{C}.
  \end{equation}
  The indicator ensures that the process is stopped upon exiting the
  domain or reaching minimal energy and the semigroup $(T_\ell)_{\ell
    \geq 0}$ satisfies:
  \begin{equation}
    T_0 = I, \quad T_{\ell_1 + \ell_2} = T_{\ell_1} T_{\ell_2}, \quad \lim_{\ell \downarrow 0} \|T_\ell v - v\|_{\cbH} = 0, \quad \forall v \in \cbH_0^+.
  \end{equation}
  Hence $(T_\ell)$ is a $C_0$-semigroup on $\cbH_0^+$, and by the
  Hille-Yosida theorem \cite[§4.4]{brezis2011functional}, it admits a
  densely defined infinitesimal generator $-\bL$ on $\cbH_0^+$, given
  by
  \begin{equation}
    -\bL v := \lim_{\ell \downarrow 0} \frac{T_\ell v - v}{\ell}, \quad \text{in } \cV.
  \end{equation}
  By Lemma~\ref{lemma:closure_and_ker_0}, $\bL$ agrees with the
  variational backward BFP operator, and is closed, monotone and
  maximal.

  Setting $u(\ell) := T_\ell v$, we have
  \begin{equation}
    u \in C([0, \infty); \cbH) \cap C^1((0, \infty); \cV),
  \end{equation}
  and $u$ satisfies the evolution equation
  \begin{equation}
    \partial_\ell u(\ell) + \bL u(\ell) = 0, \qquad u(0) = v.
  \end{equation}
  In variational form, for all $\phi \in \cH_0^-$ and almost every $\ell > 0$,
  \begin{equation}
    \label{transitiondensity2}
    \ltwop{\partial_\ell u(\ell)}{\phi} + \ltwop{\bL u(\ell)}{\phi} = 0.
  \end{equation}
  This completes the proof.
\end{proof}

Theorem~\ref{C0} provides an $\leb{2}(\mathcal{C})$ formulation of
standard results in stochastic analysis
\cite[Chapter~8]{Oksendal2003}, which are more commonly stated in an
$\leb{\infty}(\mathcal{C})$ setting. In such formulations, the
infinitesimal generator of a diffusion process governs the evolution
of expectations of test functions via the Kolmogorov backward
equation. A derivation of these ideas in the context of kinetic
equations and transport theory is given in \cite{Pavliotis2014}.

In this spirit, when a transition density $p(\ell, X, Z)$ on $[0,
  \infty) \times \mathcal{C} \times \mathcal{C}$ exists, we may write
  the semigroup as
  \begin{equation}
    T_\ell v(X) = \int_{\mathcal{C}} p(\ell, X, Z)\, v(Z)\, \dd Z, \qquad \forall X \in \mathcal{C}.
  \end{equation}
  By testing equation~\eqref{transitiondensity2} against Dirac
  sequences, we obtain the weak form of the Kolmogorov backward
  equation for the transition density
  \begin{equation}
    \partial_\ell p + \bL_X p = 0, \quad \text{in } \cV,
    \label{BP}
  \end{equation}
  where the subscript $X$ indicates that $\bL$ acts on the first
  spatial argument of $p(\ell, X, Z)$.

  In the distributional sense, this formulation enforces the initial condition
  \begin{equation}
    p(0, X, Z) = \delta_X(Z),
  \end{equation}
  as well as the absorbing boundary condition
  \begin{equation}
    p(\ell, X, Z) = 0, \quad \text{for } X \in \Gamma^+,~ \ell > 0.
  \end{equation}
  
  \begin{theorem}[Duality of the Semigroups]
    \label{lemma:semigroup}
    Suppose that the transition probability density $p(\ell, X, Z)$ of
    the process $Y_\ell$ exists with respect to the base measure on
    $\mathcal{C}$. Then, for fixed $X \in \mathcal{C}$, the map $Z
    \mapsto p(\ell, X, Z)$ satisfies the Kolmogorov forward
    (Fokker--Planck) equation
    \begin{equation}
      \partial_\ell p + \fL_Z p = 0, \quad \text{in } \cV,
      \label{FP}
    \end{equation}
    with initial condition
    \begin{equation}
      p(0, X, Z) = \delta_X(Z),
    \end{equation}
    and absorbing boundary condition
    \begin{equation}
      p(\ell, X, Z) = 0, \quad \text{for } Z \in \Gamma^+,~ \ell > 0.
      \label{zeroZ}
    \end{equation}
  \end{theorem}
  
  \begin{proof}
%
Since both $\bL$ and $\fL$ are densely defined, closed, monotone, and
maximal on $\cbH_0^+$ and $\cH_0^-$ respectively, they generate
strongly continuous semigroups $\exp(-s\bL)$ and $\exp(-s\fL)$; see
\cite{pazy2012semigroups, engel2000one}. These semigroups satisfy the
duality identity
\begin{equation}
  \ltwop{\psi}{\exp(-s \bL) \phi} = \ltwop{\exp(-s \fL) \psi}{\phi}, \qquad \forall \phi \in \cbH_0^+,~ \psi \in \cH_0^-.
  \label{semigroupduality}
\end{equation}
This follows by expanding the exponential as a power series in $s$,
applying \eqref{dualinnerproduct}, and noting that the boundary
conditions are preserved under the domain definitions of $\fL$ and
$\bL$.

Assuming a transition density $p(\ell, X, Z)$ exists, for $\psi \in \cH_0^-$, by differentiating $\ell\mapsto \langle\psi, \exp(-\ell \bL) \phi\rangle $, we note that, on the one hand, 
\begin{equation}
 \langle\psi, \exp(-\ell \bL) \phi\rangle  = \int_{\mathcal{C}}\int_{\mathcal{C}}\psi(X) p(\ell, X, Z) \phi(Z) \, \dd X\dd Z.
\end{equation}
On the other hand,  applying
\eqref{semigroupduality} and differentiating again with respect to $\ell$ tells us that 
\begin{equation}
  \partial_\ell p(\ell, X, Z) + \fL_Z p(\ell, X, Z) = 0,
\end{equation}
in the distributional sense. The initial condition follows from $T_0 =
I$, and the boundary condition $p(\ell, X, Z) = 0$ for $Z \in
\Gamma^+$ and $\ell > 0$ reflects the absorption of the forward
process at the outflow boundary.

The value $p(\ell, X, X)$ may be strictly positive for $\ell > 0$
depending on the recurrence properties of the process and does not
vanish in general.
\end{proof}

The adjoint relationship between $\bL$ and $\fL$ provides a rigorous
explanation for why the deterministic Boltzmann--Fokker--Planck
equation governs proton fluence in phase space. The backward operator
$\bL$ generates the evolution of expectations of path-dependent
functionals, while its adjoint $\fL$ governs the forward evolution of
proton densities and fluence.

This distinction underpins the roles of the two formulations in proton
radiotherapy. The backward equation is used to compute expected
observables, such as dose or track-weighted quantities, by propagating
initial values backward along the stochastic trajectories. In contrast, the
forward equation propagates phase-space distributions and is naturally
aligned with PDE-based treatment planning.

This duality forms the mathematical foundation for hybrid
computational approaches that combine Monte Carlo simulation with
deterministic transport solvers. It allows us to translate between
pathwise and density-based representations, a capability we will
exploit in the optimisation framework developed in the following
sections.

\section{Deterministic and Stochastic Resolvents}
\label{sec:resolvent}

In the deterministic framework, the operators $\fL : \cH_0^- \to \cV$
and $\bL : \cbH_0^+ \to \cV$ are unbounded, closed and maximal
monotone. Their resolvent operators are defined for $\lambda \geq 0$
by
\begin{equation}
  \fR_\lambda := (\lambda I + \fL)^{-1}, \qquad \bR_\lambda := (\lambda I + \bL)^{-1},
\end{equation}
and correspond to variational solutions of the stationary equations
\begin{equation}
  \lambda u + \bL u = v, \qquad u \in \cbH_0^+,
  \label{eq:resolvent_backward}
\end{equation}
\begin{equation}
  \lambda \phi + \fL \phi = v, \qquad \phi \in \cH_0^-.
  \label{eq:resolvent_forward}
\end{equation}
These equations are well-posed even when $\lambda = 0$, due to the
strict monotonicity of $\bL$ and $\fL$, as established in
Lemma~\ref{lemma:monotone}, and the maximality proven in
Lemma~\ref{lemma:closure_and_ker_0}. This ensures that $\fR_\lambda$
and $\bR_\lambda$ are everywhere-defined bounded linear operators $\cV
\to \cH_0^-$ and $\cV \to \cbH_0^+$, respectively.

Equation~\eqref{eq:resolvent_backward} represents the expected
cumulative contribution of a source term $v$ along stochastic proton
tracks, and admits a probabilistic interpretation as discussed
below. Equation~\eqref{eq:resolvent_forward} describes the propagation
of fluence to a point $X \in \mathcal{C}$ from a specified inflow
distribution. For homogeneous boundary conditions, the forward problem
is defined in $\cH_0^-$, which corresponds to zero incoming fluence on
$\Gamma^-$. But, in general, inhomogeneous boundary data $g \in
\operatorname{tr}(\Gamma^-)$ must be lifted into $\cH$ and subtracted
out, as explained in Section~\ref{subsec:functional-setup}.

\subsection*{Stochastic Interpretation of the Backward Resolvent}

In the stochastic formulation, the resolvent $\bR_\lambda$ of the
generator $\bL$ has a natural interpretation as the Laplace transform
of the semigroup generated by the SDE \eqref{eq:sde-tracklength}. For
$v \in \cV$ and $X \in \mathcal{C}$,
\begin{equation}
  \bR_\lambda [v](X) = \mathbb{E}_X \left[ \int_0^\infty e^{-\lambda \ell} v(Y_\ell) \mathbf{1}_{\ell < \Lambda} \, \dd \ell \right],
\end{equation}
where $\Lambda$ is the stopping time at which the particle exits the
spatial domain or reaches minimal energy. This expression represents
the $\lambda$-discounted mean occupation time of the process $Y_\ell$
weighted by the test function $v$. 

This interpretation provides a direct connection between the
deterministic resolvent equation and the sampling structure of
time-discounted Monte Carlo estimates. In particular, the
representation of $\bR_\lambda v$ as a Laplace transform of the
expected occupation integral aligns closely with classical results in
stochastic analysis, where resolvent operators are understood as
expectations of occupation measures over the process lifetime.

For any bounded measurable function $v$ vanishing on the cemetery
state $\dagger$, the backward resolvent has a probabilistic
representation
\begin{equation}
  \bR_\lambda[v](X)
  =
  \mathbb{E}_X\left[ \int_0^\infty e^{-\lambda \ell} v(Y_\ell) \mathbf{1}_{\ell < \Lambda} \, \d \ell \right], \qquad X \in \mathcal{C},
  \label{probabilityresolvent}
\end{equation}
where $\Lambda$ is the stopping time at which the process exits the
domain or reaches the minimum energy $E_{\min}$. The integrand
vanishes for $\ell \ge \Lambda$ due to the cemetery convention
$v(\dagger) = 0$, and the representation automatically enforces the
boundary condition $\bR_\lambda[v] = 0$ on $\Gamma^+$.

Since the transport process is non-conservative, the integral remains
finite even when $\lambda = 0$. The unregularised resolvent
\begin{equation}
  \bR[v](X)
  =
  \mathbb{E}_X\left[ \int_0^\infty v(Y_\ell) \mathbf{1}_{\ell < \Lambda} \, \d \ell \right]
  \label{resolventdef}
\end{equation}
is therefore well-defined and describes the expected total occupation
time in configuration space. This coincides with the variational
solution to $\bL u = v$ with $u|_{\Gamma^+} = 0$, justified
analytically by the maximal monotonicity of $\bL$.

\begin{theorem}[Existence of the Resolvent Density via Hypoellipticity]
  Suppose that the stochastic process $Y_\ell = (\vec{x}_\ell,
  \omega_\ell, E_\ell)$ is Malliavin non-degenerate in the sense that
  its Malliavin covariance matrix is almost surely
  invertible. Additionally, assume that the transition probabilities
  of $Y_\ell$ are absolutely continuous with respect to the Lebesgue
  measure on $\mathcal{C}$. Then, for $\lambda \geq 0$, the resolvent measure
  \begin{equation}
    \bR_\lambda(X; \dd \vec x' \dd \omega' \dd E'), \qquad X = (\vec x, \omega, E) \in \mathcal{C},
  \end{equation}
  admits a density $\rr_\lambda(X, (\vec x', \omega', E'))$ such that
  \begin{equation}
    \bR_\lambda(X; \dd \vec x' \dd \omega' \dd E') = 
    \rr_\lambda(X, (\vec x', \omega', E')) \, \dd \vec x' \dd \omega' \dd E'.
  \end{equation}
  As a consequence, the resolvent operator admits the kernel representation
  \begin{equation}
    \bR_\lambda[v](X) = \int_{\mathcal{C}} \rr_\lambda(X, Z) \, v(Z) \, \dd Z, \qquad X \in \mathcal{C}, 
  \end{equation}
  for any $v \in \leb{1}(\mathcal{C})$. In particular, the formula holds for $v \in \leb{2}(\mathcal{C})$, which is the natural setting in the variational framework.
\end{theorem}

\begin{proof}
  By definition, the resolvent measure satisfies
  \begin{equation}
    \bR_\lambda(X; A)
    =
    \mathbb{E}_{X} \left[ \int_0^\infty \mathbf{1}_{Y_\ell \in A} \, \mathbf{1}_{\ell < \Lambda} \, e^{-\lambda \ell} \, \dd \ell \right], \qquad X \in \mathcal{C},
  \end{equation}
  for any measurable set $A \subset \mathcal{C}$. This expresses the
  expected occupation time of the process $Y_\ell$ in the set $A$,
  discounted by $e^{-\lambda \ell}$ and stopped at the exit time
  $\Lambda$.

  Under the Malliavin non-degeneracy assumption, the process $Y_\ell$
  admits a smooth transition density $p(\ell, X, Z)$ for each $\ell >
  0$, satisfying
  \begin{equation}
    \mathbb{P}_X(Y_\ell \in A,~ \ell < \Lambda) = \int_A p(\ell, X, Z) \, \dd Z,
  \end{equation}
  for all measurable $A \subset \mathcal{C}$. The existence and
  regularity of $p$ follow from H\"ormander's theorem, owing to the
  angular diffusion term in the generator $\bL$, which ensures
  hypoellipticity \cite{hormander1967hypoelliptic,
    bismut1981martingales}.

  Substituting this into the resolvent expression, we obtain
  \begin{equation}
    \bR_\lambda(X; \dd Z) 
    = \int_0^\infty e^{-\lambda \ell} p(\ell, X, Z) \, \dd \ell \, \dd Z.
  \end{equation}
  Hence, by Fubini’s theorem, we define the density
  \begin{equation}
    \rr_\lambda(X, Z) := \int_0^\infty e^{-\lambda \ell} p(\ell, X, Z) \, \dd \ell.
  \end{equation}
  For any $v \in \leb{1}(\mathcal{C})$, we obtain
  \begin{equation}
    \bR_\lambda[v](X) = \int_{\mathcal{C}} \rr_\lambda(X, Z) \, v(Z) \, \dd Z.
  \end{equation}
  This proves the existence of a jointly measurable integral kernel
  $\rr_\lambda$ and the corresponding kernel representation of the
  resolvent operator. If higher regularity is required, one may take
  $\phi \in \sobh{k}(\mathcal{C})$.
\end{proof}

To make a formal connection between the resolvent kernel
$\rr_\lambda$ and the elliptic problem \eqref{flowfrom}, we begin
with the backward Kolmogorov equation \eqref{BP} for the transition
density $p(\ell, X, Z)$ and its exponentially weighted form. Define
$q_\lambda(\ell, X, Z) := e^{-\lambda \ell} p(\ell, X, Z)$. Then
\begin{equation}
  \partial_\ell q_\lambda(\ell, X, Z)
  = -\lambda e^{-\lambda \ell} p(\ell, X, Z) + e^{-\lambda \ell} \partial_\ell p(\ell, X, Z)
  = -\lambda q_\lambda(\ell, X, Z) - \bL_X q_\lambda(\ell, X, Z).
  \label{diffback}
\end{equation}
Hence, $q$ satisfies the evolution equation
\begin{equation}
  \partial_\ell q_\lambda(\ell, X, Z) + (\lambda I + \bL_X) q_\lambda(\ell, X, Z) = 0,
\end{equation}
in the distributional sense. Integrating over $\ell \in (0, \infty)$ and using
\begin{equation}
  \rr_\lambda(X, Z) = \int_0^\infty q_\lambda(\ell, X, Z) \, \dd \ell,
\end{equation}
we obtain
\begin{equation}
  \int_0^\infty \partial_\ell q_\lambda(\ell, X, Z) \, \dd \ell
  + (\lambda I + \bL_X) \rr_\lambda(X, Z)
  = \lim_{\ell \to \infty} q_\lambda(\ell, X, Z) - q_\lambda(0, X, Z) + (\lambda I + \bL_X) \rr_\lambda(X, Z).
\end{equation}
Since the process $Y_\ell$ is dissipative and $p(\ell, X, Z) \to 0$ as
$\ell \to \infty$, we conclude
\begin{equation}
  (\lambda I + \bL_X) \rr_\lambda(X, Z) = \delta_X(Z).
  \label{firstdelta}
\end{equation}
This shows that the resolvent kernel $\rr_\lambda(X, Z)$ is the
Green's function for the elliptic operator $(\lambda I + \bL_X)$,
subject to the natural terminal condition $\rr_\lambda(X, Z) = 0$ for
$X \in \Gamma^+$.

The reader will also note that, providing $S>0$, the solution to the
SDE \eqref{BP} always immediately moves away from its initial state
due to depleting energy and kinetic motion. This means that we
additionally have in this setting that $r_\lambda(X,Z) = 0$ for $Z\in
\Gamma^-$.

A similar argument replacing the role of $q_\lambda(\ell, X, Z)$ in
\eqref{diffback} by $\int_\mathcal{C}q_\lambda(\ell, X, Z)v(X)\dd Z$
shows that
\begin{equation}
  (\lambda I + \bL_X)  \bR_\lambda[v](X)= v(X)
  \label{integratedfirstdelta}
\end{equation}
with $ \bR_\lambda[v](X) = 0$ for $X\in\Gamma^+$.  This also justifies
the pre-emptive choice of notation $\bR_\lambda := (\lambda I +
\bL)^{-1}$.

Whilst the parameter $\lambda > 0$ traditionally plays the role of an
exponential killing rate, appearing as the Laplace dual variable in
resolvent representations of the semigroup, it is natural in this
setting to consider the case $\lambda = 0$. Due to the
non-conservative nature of the underlying stochastic process (which
terminates at finite track length with probability one), the resolvent
remains well-defined even when $\lambda = 0$. This case corresponds to
unweighted occupation time and captures the true lifetime-integrated
response of the transport system without artificial discounting.

Next, define for $\lambda\geq0$,
\begin{equation}
    \fR_\lambda[\phi](Z) = \int_{\mathcal{C}}\phi(X) \rr_\lambda(X, Z) \,  \, \dd X,
    \label{forwardR}
\end{equation}
for any $\phi \in \leb{1}(\mathcal{C})$.
An analogous argument applied to the forward equation \eqref{FP}
yields
\begin{equation}
  (\lambda I + \fL_Z) \rr_\lambda(X, Z) = \delta_X(Z),
  \label{seconddelta}
\end{equation}
subject to the condition $\rr_\lambda(X, Z) = 0$ for $Z \in \Gamma^+$,
as inherited from \eqref{zeroZ}.  In both \eqref{seconddelta} and
\eqref{firstdelta}, the distributional delta on the right-hand side
encodes the identity action of the resolvent on test functions. The
adjoint structure ensures consistency between the forward and backward
Green's functions.  As with \eqref{integratedfirstdelta}, we can
similarly show that
\begin{equation}
(\lambda I + \bL_Z)  \fR_\lambda[\phi](Z)= \phi(Z)
\label{integratedfirstdeltaf}
\end{equation}
with $ \fR_\lambda[\phi](Z) = 0$ for $Z\in \Gamma^-$.  Again, we are
led to understand our pre-emptive choice of notation $\fR_\lambda :=
(\lambda I + \fL)^{-1}$.

If we enforce $v\in\cbH_0^+\cap\cV$ and $\phi\in\cH_0^-\cap\cV$, then
we see the link between \eqref{eq:resolvent_backward} and
\eqref{integratedfirstdelta}, as well as \eqref{eq:resolvent_forward}
and \eqref{integratedfirstdeltaf}. Moreover,
\begin{equation}
\langle\phi, \bR_\lambda[v]\rangle = \langle  \phi , (\lambda I + \bL)^{-1}  v\rangle = \langle (\lambda I + \fL)^{-1} \phi, v\rangle = \langle  \fR_\lambda[\phi] , v\rangle, 
\end{equation}
where we have used the duality of $\bL$ and $\fL$, showing the natural
duality of the operators $\bR_\lambda$ and $\fR_\lambda$, for
$\lambda\geq0$.

\begin{remark}(Resolvent and kernel extension to $\Gamma^-$)
  \rm\label{extend} We note that, in either forward or backward form,
  $r_\lambda(X,Z)$ should be thought of as the ($\lambda$-discounted
  if $\lambda>0$) density of fluence incoming at $Z$ generated by an
  infinitesimal outgoing source at $X$, informally a ``flux
  capacitor'' accumulating expected flux along tracks. Importantly, we
  also note that many of the above identities can be extended by
  allowing $X\in \Gamma^-$. Indeed $r_\lambda(X,Y)$ is perfectly well
  defined for $X\in \Gamma^-$, thanks to the directional nature of the
  SDE \eqref{eq:sde-tracklength}.  That is, we can initiate the SDE
  from the configuration $X\in \Gamma^-$ and it will immediately enter
  $\mathcal{C}$ and spend a positive amount of time before exiting on
  $\Gamma^+$; as such the semigroup \eqref{semigroupu} and hence the
  associated resolvent \eqref{resolventdef} are well
  defined. Extending to $X\mapsto r_\lambda(X,Y)$ to the boundary
  $\Gamma^-$ also allows us to extend the notion of
  $\fR_\lambda[\phi]$ in \eqref{forwardR} to accommodate for functions
  $\phi$ which are supported on the entrance boundary $\Gamma^-$.

The extension to the inflow boundary can be explained as consistent
with the formulation of
\begin{equation}
  \lambda \psi + \fL \psi = \phi, \qquad \phi \in \cV.
\end{equation}
This is done by writing
\begin{equation}
  \psi = u + \widehat{\phi}\qquad \text{ such that } \qquad \lambda  u+ \fL u = - (\lambda + \fL) \widehat{g} \quad \text{in } \mathcal{C} 
\end{equation}
such that $\widehat{\phi} = \phi$ on $\Gamma^-$ and $u\in \cH_0^-$, in
particular, $u = 0$ on $\Gamma^-$.
\end{remark}

\subsection*{Dose Computation}

In both the stochastic and deterministic settings, the dose represents
the total energy deposited per unit mass in the absorbing medium. The
fundamental relation governing dose is the classical fluence-dose
formula
\begin{equation}
  \label{textbookdose}
  \text{dose} = \text{fluence} \times \text{mass stopping power}.
\end{equation}
In our formulation, this identity emerges naturally from both the
stochastic and PDE-based perspectives, via a common dose operator
acting on a notion of fluence.

\subsection{Dose in the Stochastic Framework}

We are interested 
in viewing dose from a purely spatial perspective. If $S(\vec{x}, E)$
is the stopping power and $\rho(\vec x)$ is the mass density, then the
expected contribution to dose from a single proton starting at $X \in
\Gamma^-$, as observed through a test function $v$ on $\mathcal{C}$,
is
\begin{equation}
\bR[vS/\rho](X) =   \int_{\mathcal{C}}  \rr_0(X, (\vec{x}, \omega, E)) \,v(\vec{x}, \omega, E) \, \frac{S(\vec{x}, E)}{\rho(\vec{x})} \, \dd \vec{x} \dd \omega \dd E,
  \label{dose-stochastic}
\end{equation}
where we have taken advantage of Remark \ref{extend}.
If there is inflow boundary data $g$ on
$\Gamma^-$, then the the dose profile is given by 
\begin{equation}
\langle g, \bR[vS/\rho]\rangle = \int_{\mathcal{C}} \rr_0^{(g)}(\vec{x}, \omega, E) \,v(\vec{x}, \omega, E) \, \frac{S(\vec{x}, E)}{\rho(\vec{x})} \, \dd \vec{x} \dd \omega \dd E
\end{equation} 
where 
\begin{equation}
  \rr_0^{(g)}(\vec{x}, \omega, E) = \int_{\Gamma^-} g(X)\rr_0(X, (\vec{x}, \omega, E))\,\dd X
\end{equation}
so that the pointwise dose density is
\begin{equation}
  \rr_0^{(g)}(\vec{x}, \omega, E) \, \frac{S(\vec{x}, E)}{\rho(\vec{x})},
  \label{dose-density}
\end{equation}
thus conforming to the paradigm that
\begin{equation}
  \text{dose} = \text{fluence} \times \text{mass stopping power}.
\end{equation}
Finally, the spatial dose $\vec x \mapsto \mathcal{D}(\vec x)$
received by the medium at position $\vec x$ is given by marginalising
over angle and energy
\begin{equation}
  \label{spatialdose}
  \mathcal{D}(\vec x) = \int_{\mathcal{E}} \frac{S(\vec{x}, E)}{\rho(\vec{x})} \int_{\mathbb{S}^2} \rr^{(g)}_0(\vec{x}, \omega, E) \, \dd \omega \dd E.
\end{equation}

\subsection{Dose in the Deterministic Framework}

In the deterministic setting, the fluence $\psi$ satisfies the forward
Boltzmann--Fokker--Planck equation
\begin{equation}
 \fL \psi  = g\quad \text{in } \mathcal{C}
 \label{linktorg}
\end{equation}
 with inflow boundary data $g$ on
$\Gamma^-$ {(cf. Remark \ref{extend})}.

 The dose is then obtained by integrating the fluence against the
energy loss per unit mass
\begin{equation}
  \mathcal{D}(\vec{x}) 
  = \int_{\mathcal{E}} \frac{S(\vec{x}, E)}{\rho(\vec{x})} \int_{\mathbb{S}^2} \psi(\vec{x}, \omega, E) \, \dd \omega \dd E
  \label{dose-deterministic}
\end{equation}
In differential form, the dose density is
\begin{equation}
  \psi(\vec{x}, \omega, E) \, \frac{S(\vec{x}, E)}{\rho(\vec{x})},
  \label{psidensity}
\end{equation}
reaffirming the fundamental relation
\begin{equation}
  \text{dose} = \text{fluence} \times \text{mass stopping power}.
\end{equation}
The two calculations for spatial dose are equivalent because, as per \eqref{integratedfirstdeltaf} with $\lambda = 0$, the solution to \eqref{linktorg} is identified as 
\begin{equation}
  \psi(\vec{x}, \omega, E) = \fR_\lambda[g](\vec{x}, \omega, E) = \int_{\Gamma^-}
  g(X) \rr_0(X, (\vec{x}, \omega, E))\,\dd X = \rr^{(g)}_0(\vec{x}, \omega, E).
\end{equation}
where we have again taken advantage of Remark \ref{extend}, and hence \eqref{psidensity} agrees with 
\eqref{dose-density}.

\section{Hybrid SDE--PDE Optimisation for Treatment Planning}
\label{sec:treatment}

In this section, we propose a proof-of-concept framework for treatment
planning in proton therapy that leverages the dual
stochastic-deterministic formulation developed in this work. Our goal
is to show how the SDE model of individual proton transport and the
PDE-based variational formulation can be combined to guide algorithm
design and theoretical analysis under uncertainty.

The stochastic formulation captures the physical randomness of elastic
scattering, energy loss, and non-elastic interactions through an SDE
for the full proton life cycle, simulated via Monte Carlo methods. The
deterministic framework, by contrast, provides a variational structure
suited to adjoint-based sensitivity analysis and optimisation, via the
forward Boltzmann--Fokker--Planck equation and its resolvent.

Our hybrid strategy exploits this duality, dose is computed using
high-fidelity stochastic simulations for fixed configurations, while
gradients are evaluated efficiently using deterministic adjoint
PDEs. This enables scalable control under uncertainty while preserving
physical accuracy.

Uncertainty in treatment planning arises not only from particle-scale
randomness but also from anatomical and physiological variability,
modelled by a parameter $\theta$. For each realisation $\theta$, the
stochastic and deterministic models agree; however, the optimisation
must be robust across the distribution of $\theta$. This induces a
nested structure, the inner expectation captures particle variability
for fixed $\theta$, while the outer expectation averages dose over
scenario uncertainty.

This structure motivates a class of hybrid algorithms combining
forward Monte Carlo solvers with deterministic adjoint methods. Our
formulation aligns with recent efforts to incorporate uncertainty into
radiotherapy planning \cite{bangert2013analytical}, while remaining
compatible with classical constrained optimisation approaches such as
\cite{bortfeld1993optimization}.

\subsection{Problem Formulation}
\label{subsec:problem-formulation}

We aim to determine an optimal proton beam intensity distribution $g:
\Gamma^- \to \mathbb{R}_+$, prescribing the number of protons injected
at each position $\vec{x}$, direction $\omega$ and energy $E$ on the
inflow boundary of phase space. The goal is to select $g$ such that
the resulting dose approximates a prescribed clinical target while
respecting physical and biological constraints.

The delivered fluence $\psi_g^\theta$ depends not only on $g$, but
also on a random parameter $\theta \in \Theta$, which models
uncertainty in anatomical and physical characteristics. These include
patient-specific variations in tissue density $\rho^\theta(\vec{x})$,
stopping power $S^\theta(\vec{x}, E)$, and domain geometry.

For each scenario $\theta$, the forward transport problem yields a
fluence $\psi_g^\theta$, from which dose is computed via the
scenario-dependent operator
\begin{equation}
  \mathcal{D}_\theta[\psi_g^\theta](\vec{x}) =
  \int_{\mathcal{E}} \frac{S^\theta(\vec{x}, E)}{\rho^\theta(\vec{x})}
  \int_{\mathbb{S}^2} \psi_g^\theta(\vec{x}, \omega, E) \, \dd \omega \dd E.
\end{equation}
As $\theta$ is random, we consider the expected dose
\begin{equation}
  \mathbf{E}_\theta\left[ \mathcal{D}_\theta[\psi_g^\theta](\vec{x}) \right] =
  \int_\Theta \mathcal{D}_\theta[\psi_g^\theta](\vec{x}) \, \mathrm{d}\mathbb{P}(\theta).
\end{equation}
In general, the dose operator does not commute with expectation:
\begin{equation}
  \mathbf{E}_\theta\left[ \mathcal{D}_\theta[\psi_g^\theta] \right]
  \neq \mathcal{D} \left[ \mathbf{E}_\theta[\psi_g^\theta] \right].
\end{equation}
This non-commutativity arises because both $\psi_g^\theta$ and the
operator $\mathcal{D}_\theta$ depend on $\theta$.

Consequently, the expected dose must be approximated using Monte Carlo
sampling. Drawing scenarios $\{\theta^{(i)}\}_{i=1}^N$, solving the
transport equation for each $\theta^{(i)}$, computing the
corresponding dose and averaging yields an empirical estimate of
$\mathbf{E}_\theta[\mathcal{D}_\theta[\psi_g^\theta]]$. This structure
underpins robust treatment planning across uncertain physical and
anatomical parameters \cite{georgiou2025scotty}.

We impose box constraints on the control to reflect physical bounds on
beam delivery
\begin{equation}
  0 \leq g(\vec{x}, \omega, E) \leq g_{\max}(\vec{x}, \omega, E), \qquad \text{a.e. on } \Gamma^-,
\end{equation}
where $g_{\max} \in \leb{\infty}(\Gamma^-)$ encodes machine-specific
limits. The admissible control set is
\begin{equation}
  \mathcal{U}_{\mathrm{ad}} :=
  \left\{ g \in \leb{2}(\Gamma^-) ~\middle|~ 0 \leq g \leq g_{\max} \text{ a.e. on } \Gamma^- \right\},
\end{equation}
a closed, convex subset of $\leb{2}(\Gamma^-)$.

In the deterministic setting, the dose depends only on $g$, and we
define the cost functional
\begin{equation}
  \label{eq:cost}
  \mathcal{J}(g; \theta) :=
  \int_D w(\vec{x}) \left( \mathcal{D}_\theta[\psi_g^\theta](\vec{x}) - d_{\text{prescribed}}(\vec{x}) \right)^2 \dd \vec{x}
  + \frac{\alpha}{2} \| g \|^2_{\leb{2}(\Gamma^-)},
\end{equation}
where $d_{\text{prescribed}}$ is the target dose, $w(\vec{x})$ encodes
clinical priorities (e.g., tumour vs. healthy tissue) and $\alpha > 0$
is a regularisation parameter penalising large intensities.

When uncertainty is explicitly modelled, both $\psi_g^\theta$ and
$\mathcal{D}_\theta$ depend on the random parameter $\theta \in
\Theta$. The expected cost becomes
\begin{equation}
  \label{eq:expectedcost}
  \mathbf{E}_\theta[\mathcal{J}(g; \theta)] :=
  \mathbf{E}_\theta \left[ \int_D w(\vec{x}) \left( \mathcal{D}_\theta[\psi_g^\theta](\vec{x}) - d_{\text{prescribed}}(\vec{x}) \right)^2 \dd \vec{x} \right]
  + \frac{\alpha}{2} \| g \|^2_{\leb{2}(\Gamma^-)}.
\end{equation}
This formulation corresponds to risk-neutral optimisation, the control
$g$ is deterministic, but the objective penalises the \emph{expected}
mismatch between delivered and prescribed dose. Alternative
strategies, such as penalising variance or minimising worst-case
deviation, lie within the broader class of robust optimisation. See
\cite{georgiou2025scotty} for a comparison of such models.

\begin{lemma}[Interchange of G\^ateaux Derivative and Expectation]
  \label{lemma:interchange}
  Let $\Theta$ be a probability space and suppose that for each
  $\theta \in \Theta$, the map $g \mapsto \mathcal{J}(g; \theta)$ is
  G\^ateaux differentiable on a Banach space $\mathcal{X}$ with
  derivative $\mathcal{J}'(g; \theta) \in \mathcal{X}^*$. Assume that
  \begin{enumerate}
    \item For each $g \in \mathcal{X}$ and $h \in \mathcal{X}$, the map
    $\theta \mapsto \langle \mathcal{J}'(g; \theta), h \rangle$ is
    measurable;
    \item There exists an integrable function $M \in L^1(\Theta)$ such
    that for all $g$ in a neighbourhood of $g_0$ and all unit vectors $h \in \mathcal{X}$,
    \begin{equation}
      \left| \langle \mathcal{J}'(g; \theta), h \rangle \right| \leq M(\theta) \qquad \text{for a.e. } \theta \in \Theta.
    \end{equation}
  \end{enumerate}
  Then, the expected functional $\mathbf{E}_\theta[\mathcal{J}(g; \theta)]$ is Gâteaux differentiable, and its derivative satisfies
  \begin{equation}
    \left\langle \mathcal{J}'(g), h \right\rangle
    :=
    \frac{\dd}{\dd\varepsilon} \Big|_{\varepsilon = 0} \mathbf{E}_\theta[\mathcal{J}(g + \varepsilon h; \theta)]
    =
    \mathbf{E}_\theta \left[ \left\langle \mathcal{J}'(g; \theta), h \right\rangle \right].
  \end{equation}
\end{lemma}

The box constraints $0 \leq g \leq g_{\max}$ reflect physical
limitations on beam intensity delivery imposed by the accelerator
hardware. These constraints define a closed and convex admissible set
$\mathcal{U}_{\mathrm{ad}} \subset \leb2(\Gamma^-)$, which is weakly
sequentially closed. This structural property ensures the existence of
minimisers via the direct method in the Calculus of Variations.

In the optimality conditions, the box constraints give rise to a
variational inequality that characterises admissible descent
directions. Equivalently, the gradient step for $g$ must be projected
onto $\mathcal{U}_{\mathrm{ad}}$, leading to a projected gradient or
semi-smooth update rule. This formulation naturally accommodates
practical control limits within the optimisation algorithm.

\begin{theorem}[Existence of an Optimal Control]
  Let $\alpha > 0$, and suppose the stopping power $S \in
  W^{1,\infty}(D \times \mathcal{E})$, the density $\rho \in
  L^\infty(D; (0,\infty))$, and the spatial weighting function $w \in
  L^\infty(D)$. Assume
  \begin{enumerate}
  \item For each $g \in \mathcal{U}_{\mathrm{ad}}$, the stochastic
    forward problem admits a unique fluence $\psi_g^\theta \in \cH$
    for almost every $\theta \in \Theta$;
  \item The mapping $g \mapsto \psi_g^\theta$ is linear and continuous
    from $\leb2(\Gamma^-)$ into $\cH$, uniformly in $\theta$;
  \item The dose operator $\mathcal{D}: \cH \to L^2(D)$ is linear and
    bounded;
  \item For all $g \in \mathcal{U}_{\mathrm{ad}}$, the integrand   
    \begin{equation}
      \theta \mapsto \left\| \mathcal{D}[\psi_g^\theta] - d_{\mathrm{prescribed}} \right\|^2_{L^2(D; w)}
    \end{equation}
    is integrable over $\Theta$.
  \end{enumerate}
  Then, the expected cost functional
  \begin{equation}
    \mathbf{E}_\theta\left[ \mathcal{J}(g) \right] :=
    \mathbf{E}_\theta \left[ \left\| \mathcal{D}[\psi_g^\theta] - d_{\mathrm{prescribed}} \right\|^2_{L^2(D; w)} \right]
    + \frac{\alpha}{2} \| g \|^2_{\leb2(\Gamma^-)}
  \end{equation}
  is weakly lower semicontinuous on $\leb2(\Gamma^-)$ and admits at
  least one minimiser $g^\star \in \mathcal{U}_{\mathrm{ad}}$.
\end{theorem}

\begin{proof}
  Let $\{ g_n \} \subset \mathcal{U}_{\mathrm{ad}}$ be a minimising sequence
  \begin{equation}
    \lim_{n \to \infty} \mathbf{E}_\theta\left[ \mathcal{J}(g_n) \right]
    = \inf_{g \in \mathcal{U}_{\mathrm{ad}}} \mathbf{E}_\theta\left[ \mathcal{J}(g) \right].
  \end{equation}
  Since $\mathcal{U}_{\mathrm{ad}}$ is convex, closed and bounded in
  $\leb2(\Gamma^-)$, it is weakly sequentially compact. By the
  Banach-Alaoglu theorem, there exists a subsequence (still denoted
  $g_n$) and a weak limit $g^\star \in \mathcal{U}_{\mathrm{ad}}$ such
  that
  \begin{equation}
    g_n \rightharpoonup g^\star \quad \text{in } \leb2(\Gamma^-).
  \end{equation}
  By the linearity and continuity of $g \mapsto \psi_g^\theta$ and of
  $\mathcal{D}$, the composition $g \mapsto
  \mathcal{D}[\psi_g^\theta]$ is linear and continuous from
  $\leb2(\Gamma^-)$ to $L^2(D)$ for each $\theta$. Therefore, the map
  \begin{equation}
    g \mapsto \left\| \mathcal{D}[\psi_g^\theta] - d_{\mathrm{prescribed}} \right\|^2_{L^2(D; w)}
  \end{equation}
  is convex and continuous, hence weakly lower semicontinuous.

  Now, the function
  \begin{equation}
    f_n(\theta) := \left\| \mathcal{D}[\psi_{g_n}^\theta] - d_{\mathrm{prescribed}} \right\|^2_{L^2(D; w)}
  \end{equation}
  satisfies $f_n(\theta) \to f(\theta)$ pointwise a.e.\ in $\theta$,
  where $f(\theta) := \left\| \mathcal{D}[\psi_{g^\star}^\theta] -
  d_{\mathrm{prescribed}} \right\|^2$. Since each $f_n(\theta) \geq 0$
  and $\mathbf{E}_\theta[f_n]$ is bounded, Fatou's lemma yields 
  \begin{equation}
    \mathbf{E}_\theta[f] \leq \liminf_{n \to \infty} \mathbf{E}_\theta[f_n].
  \end{equation}
  Thus the expected fidelity term is weakly lower semicontinuous.

  The regularisation term $\frac{\alpha}{2} \| g \|^2_{\leb2}$ is
  convex and strongly continuous, hence also weakly lower
  semicontinuous. Summing both contributions gives
  \begin{equation}
    \mathbf{E}_\theta[ \mathcal{J}(g^\star) ] \leq \liminf_{n \to \infty} \mathbf{E}_\theta[ \mathcal{J}(g_n) ],
  \end{equation}
  and $g^\star \in \mathcal{U}_{\mathrm{ad}}$ is a minimiser.
\end{proof}

\subsection{First-Order Optimality Conditions}
\label{subsec:optimality-conditions}

We now derive the first-order necessary conditions for optimality in
the risk-neutral formulation. These conditions are expressed in terms
of the directional derivative of the cost functional and are evaluated
using the adjoint method.

For each fixed uncertainty realisation $\theta \in \Theta$, the
forward fluence $\psi_g^\theta$ solves the transport equation
\begin{equation}
  \fL^\theta \psi_g^\theta = 0 \quad \text{in } \mathcal{C}, \qquad \psi_g^\theta = g \quad \text{on } \Gamma^-,
\end{equation}
and the adjoint variable $z_g^\theta$ solves the backward transport
equation
\begin{equation}
  \bL^\theta z_g^\theta = -\frac{q S^\theta}{\rho^\theta} \left( \mathcal{D}[\psi_g^\theta] - d_{\mathrm{prescribed}} \right) \quad \text{in } \mathcal{C}, \qquad z_g^\theta = 0 \quad \text{on } \Gamma^+.
\end{equation}

These forward-adjoint pairs allow us to compute the G\^ateaux
derivative of the cost functional with respect to the control. For
each $\theta$, the directional derivative of $\mathcal{J}(g, \theta)$
in the direction $h \in \leb2(\Gamma^-)$ is given by
\begin{equation}
  \delta \mathcal{J}(g,\theta)[h] = \ltwop{z_g^\theta}{h}_{\Gamma^-} + \alpha \langle g, h \rangle_{\leb2(\Gamma^-)}.
\end{equation}
Averaging over the distribution of $\theta$ yields the G\^ateaux
derivative of the expected objective
\begin{equation}
  \nabla_g \mathbf{E}_\theta[\mathcal{J}(g)] = \mathbf{E}_\theta[z_g^\theta] + \alpha g.
\end{equation}
This characterisation provides the first-order optimality condition
for the unconstrained problem. A control $g^\star$ satisfies
\begin{equation}
  \nabla_g \mathbf{E}_\theta[\mathcal{J}(g^\star)] = 0
\end{equation}
if and only if it is a stationary point. In the presence of control
constraints, this gradient enters a variational inequality
formulation, as developed in subsequent sections.

\begin{theorem}[Gâteaux Differentiability of the Expected Cost Functional]
  Let $\alpha > 0$, and suppose that for each $\theta \in \Theta$, the
  control-to-fluence map $g \mapsto \psi_g^\theta \in \cH$ is linear
  and continuous and that the dose operator $\mathcal{D}: \cH \to
  L^2(D)$ is linear and bounded. Assume further that the integrand
  \begin{equation}
    \theta \mapsto \left\| \mathcal{D}[\psi_g^\theta] - d_{\text{prescribed}} \right\|_{L^2(D; w)}^2
  \end{equation}
  is integrable with respect to $\mathbb{P}(\theta)$ for each $g \in
  \mathcal{U}_{\mathrm{ad}}$. Then the expected cost functional
  \begin{equation}
    \mathcal{J}(g) := \mathbf{E}_\theta \left[ \int_D w(\vec{x}) \left( \mathcal{D}[\psi_g^\theta](\vec{x}) - d_{\text{prescribed}}(\vec{x}) \right)^2 \dd \vec{x} \right] + \frac{\alpha}{2} \| g \|^2_{\leb2(\Gamma^-)}
  \end{equation}
  is Gâteaux differentiable on $\mathcal{U}_{\mathrm{ad}} \subset
  \leb2(\Gamma^-)$. For any direction $h \in \leb2(\Gamma^-)$, the
  directional derivative is
  \begin{equation}
    \delta \mathcal{J}(g; h) =
    \mathbf{E}_\theta \left[ 2 \int_D w(\vec{x}) \left( \mathcal{D}[\psi_g^\theta](\vec{x}) - d_{\text{prescribed}}(\vec{x}) \right) \cdot \mathcal{D}[\psi_h^\theta](\vec{x}) \dd \vec{x} \right]
    + \alpha \langle g, h \rangle_{\leb2(\Gamma^-)}.
  \end{equation}
\end{theorem}

\begin{proof}
  Fix $g, h \in \leb2(\Gamma^-)$. For each $\theta \in \Theta$, the
  mapping $g \mapsto \mathcal{D}[\psi_g^\theta]$ is linear and
  continuous by assumption. Therefore, the inner fidelity term
  \begin{equation}
    \mathcal{J}(g, \theta) := \int_D w(\vec{x}) \left( \mathcal{D}[\psi_g^\theta](\vec{x}) - d_{\text{prescribed}}(\vec{x}) \right)^2 \dd \vec{x}
  \end{equation}
  is a quadratic functional of $g$ and hence Gâteaux differentiable
  with pointwise directional derivative
  \begin{equation}
    \delta \mathcal{J}(g, \theta; h) =
    2 \int_D w(\vec{x}) \left( \mathcal{D}[\psi_g^\theta](\vec{x}) - d_{\text{prescribed}}(\vec{x}) \right) \cdot \mathcal{D}[\psi_h^\theta](\vec{x}) \dd \vec{x}.
  \end{equation}
  Since $g \mapsto \mathcal{D}[\psi_g^\theta]$ is continuous and
  bounded uniformly over $\theta$ on bounded sets, and since
  \begin{equation}
    \left| \delta \mathcal{J}(g, \theta; h) \right| \leq C(\theta) \, \|h\|_{\leb2(\Gamma^-)}
  \end{equation}
  for an integrable function $C(\theta)$ (e.g., a bound on $\|
  \mathcal{D}[\psi_h^\theta] \|$), the conditions of
  Lemma~\ref{lemma:interchange} apply. We may therefore interchange
  expectation and differentiation
  \begin{equation}
    \delta \mathcal{J}(g; h) = \mathbf{E}_\theta[\delta \mathcal{J}(g, \theta; h)] + \alpha \langle g, h \rangle_{\leb2(\Gamma^-)}.
  \end{equation}
  This proves the Gâteaux differentiability of $\mathcal{J}$ and
  provides the desired representation.
\end{proof}

\begin{corollary}[Gradient Representation via the Adjoint Problem]
  Let $\psi_g^\theta \in \cH$ denote the fluence associated to control
  $g$ under uncertainty realisation $\theta \in \Theta$, defined by
  the forward problem
  \begin{equation}
    \fL^\theta \psi_g^\theta = 0 \quad \text{in } \mathcal{C}, \qquad \psi_g^\theta = g \quad \text{on } \Gamma^-.
  \end{equation}
  Let $z_g^\theta \in \cbH_0^+$ solve the adjoint transport equation
  \begin{equation}
    \bL^\theta z_g^\theta = -\frac{q S^\theta}{\rho^\theta} \left( \mathcal{D}[\psi_g^\theta] - d_{\mathrm{prescribed}} \right) \quad \text{in } \mathcal{C}, \qquad z_g^\theta = 0 \quad \text{on } \Gamma^+.
  \end{equation}
  Then the Gâteaux derivative of $\mathcal{J}$ satisfies
  \begin{equation}
    \delta \mathcal{J}(g; h) = \mathbf{E}_\theta \left[ \ltwop{z_g^\theta}{h}_{\Gamma^-} \right] + \alpha \langle g, h \rangle_{\leb2(\Gamma^-)}.
  \end{equation}
  In particular, the $\leb2(\Gamma^-)$-gradient is given by
  \begin{equation}
    \nabla_g \mathcal{J}(g) = \mathbf{E}_\theta [z_g^\theta] + \alpha g.
  \end{equation}
\end{corollary}

\subsection{Stochastic Gradient Estimators}
\label{subsec:stochastic-gradients}

In the hybrid formulation, the fluence $\psi_g^\theta$ is computed via
stochastic forward simulations depending on a random parameter $\theta
\in \Theta$, while the adjoint variable $z_g^\theta$ is obtained by
solving a deterministic PDE conditioned on $\theta$. This results in
stochastic variation in both the forward dose and the corresponding
adjoint source term.

To compute the expected gradient $\nabla_g
\mathbf{E}_\theta[\mathcal{J}(g)]$, we use a Monte Carlo
estimator. Let $\{ \theta_i \}_{i=1}^N$ be independent samples from
the uncertainty distribution. For each $\theta_i$:

\begin{enumerate}
  \item Solve the forward problem to obtain the fluence $\psi_g^{\theta_i}$;
  \item Evaluate the dose $\mathcal{D}[\psi_g^{\theta_i}]$;
  \item Solve the adjoint equation
  \begin{equation}
    \bL^{\theta_i} z_g^{\theta_i} = -\frac{q S^{\theta_i}}{\rho^{\theta_i}} \left( \mathcal{D}[\psi_g^{\theta_i}] - d_{\mathrm{prescribed}} \right), \qquad z_g^{\theta_i} = 0 \text{ on } \Gamma^+;
  \end{equation}
  \item Compute the adjoint trace $z_g^{\theta_i}|_{\Gamma^-}$.
\end{enumerate}

The empirical gradient estimator is then defined as
\begin{equation}
  \nabla_g \mathcal{J}_N(g) := \frac{1}{N} \sum_{i=1}^N z_g^{\theta_i} + \alpha g.
\end{equation}
As $N \to \infty$, the estimator converges strongly in
$\leb2(\Gamma^-)$ to the true gradient $\mathbf{E}_\theta[z_g^\theta]
+ \alpha g$, under mild integrability conditions on $\theta \mapsto
z_g^\theta$.

The variance of the Monte Carlo estimator $\nabla_g \mathcal{J}_N(g)$
can be substantial if the dose $\mathcal{D}[\psi_g^\theta]$ is highly
sensitive to $\theta$. Standard techniques such as importance sampling
may be employed to reduce variance \cite{hoogenboom2008zero}. More
advanced strategies, such as multifidelity corrections or Gaussian
process surrogates, are effective when forward solves are expensive
and adjoint source terms vary smoothly with respect to $\theta$.  

\subsection{Stochastic Variational Inequalities for Constrained Optimisation}
\label{subsec:stochastic-vi}

The admissible set $\mathcal{U}_{\mathrm{ad}} \subset
\leb2(\Gamma^-)$, defined by box constraints on the control $g$, is
closed and convex. In the presence of stochastic uncertainty, the
first-order optimality condition for minimising the expected cost
functional subject to $g \in \mathcal{U}_{\mathrm{ad}}$ is
approximated using a finite-sample Monte Carlo estimator of the
gradient.

Let $\nabla_g \mathcal{J}_N(g) := z_g^N + \alpha g$ denote the
empirical gradient estimator, where $z_g^N$ is the sample average of
adjoint traces computed from $N$ realisations of the uncertainty
parameter $\theta$. The stochastic variational inequality
associated with this setting is
\begin{equation}
  \label{eq:stochastic-vi}
  \ltwop{z_g^N + \alpha g}{\tilde{g} - g}_{\Gamma^-} \ge 0
  \qquad \Foreach \tilde{g} \in \mathcal{U}_{\mathrm{ad}}.
\end{equation}
This condition characterises stationarity with respect to the
empirical gradient estimate and ensures that no admissible variation
decreases the sampled cost functional.

The variational inequality \eqref{eq:stochastic-vi} leads to a
projected stochastic gradient update of the form
\begin{equation}
  g_{k+1} = \mathcal{P}_{\mathcal{U}_{\mathrm{ad}}} \left( g_k - \tau_k \nabla_g \mathcal{J}_N(g_k) \right)
  = \mathcal{P}_{\mathcal{U}_{\mathrm{ad}}} \left( g_k - \tau_k (z_{g_k}^N + \alpha g_k) \right),
\end{equation}
where $\tau_k > 0$ is the step size and
$\mathcal{P}_{\mathcal{U}_{\mathrm{ad}}}$ is the orthogonal projection
onto the admissible set in $\leb2(\Gamma^-)$. Under appropriate
assumptions on step sizes and control of the gradient variance, this
iteration converges almost surely to a solution of the true
variational inequality as $N \to \infty$.

  \label{rem:projection}
  The projection onto the admissible set is computed pointwise almost everywhere as
  \begin{equation}
    \mathcal{P}_{\mathcal{U}_{\mathrm{ad}}}[g](\vec{x}, \omega, E)
    = \min\left\{ g_{\max}(\vec{x}, \omega, E), \max\{ 0, g(\vec{x}, \omega, E) \} \right\}.
  \end{equation}
  This ensures that the control remains feasible throughout the
  iteration and enforces the physical constraints imposed by the
  treatment delivery system.

\subsection{Stochastic Realisations and Empirical Primal–Dual Systems}

In practice, the expected fluence and dose quantities in the hybrid
optimisation framework cannot be computed exactly. Instead, they are
approximated using a finite number of stochastic realisations. Let $\{
\theta^{(i)} \}_{i=1}^N$ be independent samples from the uncertainty
distribution $\mathbb{P}_\theta$, and let $\{ \psi_g^{(i)} \}_{i=1}^N$
be the corresponding forward fluence solutions, i.e., $\psi_g^{(i)} :=
\psi_g^{\theta^{(i)}}$.

We define the empirical average of these realisations as
\begin{equation}
  \bar{\psi}_g^N := \frac{1}{N} \sum_{i=1}^N \psi_g^{(i)}.
\end{equation}
The corresponding empirical dose is then given by
\begin{equation}
  \mathcal{D}[\bar{\psi}_g^N](\vec{x}) 
  := \int_{\mathcal{E}} \frac{S(\vec{x}, E)}{\rho(\vec{x})}
  \int_{\mathbb{S}^2} \bar{\psi}_g^N(\vec{x}, \omega, E) \, \dd \omega \dd E.
\end{equation}

We define the empirical adjoint $z_g^N$ as the solution to the adjoint
PDE with source determined by the empirical dose discrepancy:
\begin{equation}
  \bL z_g^N = -\frac{q S}{\rho} \left( \mathcal{D}[\bar{\psi}_g^N] - d_{\text{prescribed}} \right),
  \qquad z_g^N = 0 \quad \text{on } \Gamma^+.
\end{equation}

Together, these yield the empirical primal–dual system:
\begin{align}
  \fL \bar{\psi}_g^N &= 0, &\quad & \text{in } \mathcal{C}, 
  & \bar{\psi}_g^N &= g, & \text{on } \Gamma^-,
  \nonumber \\
  \bL z_g^N + \frac{q S}{\rho} \left( \mathcal{D}[\bar{\psi}_g^N] - d_{\text{prescribed}} \right) &= 0, &\quad & \text{in } \mathcal{C}, 
  & z_g^N &= 0, & \text{on } \Gamma^+,
  \\
  g^\star - \Pi_{[0, g_{\max}]} \left( -\frac{1}{\alpha} z_g^N \right)&= 0, &\quad & \text{on } \Gamma^-.& & & 
  \nonumber
\end{align}

\begin{remark}[Stochastic Variational Inequality from the Empirical System]
  The projection formula from Remark~\ref{rem:projection} applies to
  the empirical adjoint $z_g^N$ as well. Each sample-based
  approximation satisfies the variational inequality
  \begin{equation}
    \ltwop{z_g^N + \alpha g^\star}{\tilde{g} - g^\star}_{\Gamma^-} \geq 0 
    \quad \Foreach \tilde{g} \in \mathcal{U}_{\mathrm{ad}}.
  \end{equation}
  This characterises a sample-wise approximation of the true
  stochastic variational inequality and supports the use of empirical
  gradients in projected optimisation.
\end{remark}

  The empirical adjoint $z_g^N$ serves as a Monte Carlo gradient
  estimator for the expected fidelity functional. As $N \to \infty$,
  it converges (in probability) to the true adjoint $z_g :=
  \mathbf{E}_\theta[z_g^\theta]$, and the empirical system recovers
  the deterministic optimality conditions. In practice, small $N$
  yields efficient approximate updates and variance reduction
  techniques can be used to stabilise and accelerate convergence.

\subsection{Algorithmic Optimisation Strategy}
\label{subsec:algorithm}

We now present a numerical implementation of the hybrid SDE–PDE
framework for robust treatment planning. This algorithm combines
high-fidelity stochastic forward simulations with deterministic
adjoint-based gradient computation to solve the risk-neutral
optimisation problem under physical uncertainty.

At each iteration, the current control $g^k$ is used to generate $N$
realisations of the stochastic fluence $\psi_g^\theta$. These are
averaged to yield an empirical estimate $\bar{\psi}_{g^k}^N$ of the
expected fluence. The corresponding dose discrepancy is used to
construct the source term in the adjoint Boltzmann--Fokker--Planck
equation, which is solved deterministically to produce the empirical
gradient direction $z^k$. A projected gradient update is then applied
to ensure feasibility with respect to box constraints.

\begin{algorithm}[h!]
  \caption{Hybrid SDE-PDE Optimisation Algorithm}
  \label{alg:hybrid}
  \begin{algorithmic}[1]
    \State \textbf{Input.} Initial control $g^0 \in \mathcal{U}_{\mathrm{ad}}$, step size $\alpha > 0$, tolerance $\varepsilon > 0$, sample size $N$
    \State \textbf{Output.} Approximate minimiser $g^\star$
    
    \While{$\| z^k + \alpha g^k \|_{\leb2(\Gamma^-)} > \varepsilon$}
    
      \State \textbf{(i) Forward Solve.} Simulate $N$ samples $\{\psi^{\theta_i}_{g^k}\}_{i=1}^N$ using the SDE model with inflow $g^k$
      \State \hspace{1em} Compute the empirical average $\bar{\psi}^N_{g^k} := \frac{1}{N} \sum_{i=1}^N \psi^{\theta_i}_{g^k}$
      \State \hspace{1em} Evaluate dose $\mathcal{D}[\bar{\psi}^N_{g^k}]$
    
      \State \textbf{(ii) Adjoint Solve.} Solve
      \[
      \bL z^k = -\frac{q S}{\rho} \left( \mathcal{D}[\bar{\psi}^N_{g^k}] - d_{\text{prescribed}} \right), \qquad z^k = 0 \text{ on } \Gamma^+
      \]
    
      \State \textbf{(iii) Control Update.}
      \[
      g^{k+1} = \Pi_{[0, g_{\max}]} \left( -\frac{1}{\alpha} z^k \right)
      \]
    
      \State \textbf{(iv) Convergence Check.} If $\| z^k + \alpha g^k \|_{\leb2(\Gamma^-)} \leq \varepsilon$, \textbf{break}
    
    \EndWhile
    \State \textbf{Return:} $g^\star = g^k$
  \end{algorithmic}
\end{algorithm}

The accuracy of the gradient estimate depends on the number of Monte
Carlo samples $N$ used to compute $\bar{\psi}_{g^k}^N$. Early
iterations may tolerate high variance and hence use small $N$, while
later stages may require more accurate gradients. Adaptive sampling
schemes, in which $N$ increases with $k$ or is chosen based on
variance estimates, can improve efficiency and convergence robustness.

\subsection{Numerical Demonstration in One Dimension}
\label{subsec:numerical-1d}

To illustrate the hybrid optimisation strategy, we present a simple
numerical experiment in a one-dimensional water phantom. The spatial
domain is an $8.5\,\mathrm{cm}$ homogeneous slab of water, discretised
uniformly in depth and energy. A bank of $20$ pencil beams is
constructed with varying initial energies, each modelled as a Gaussian
energy distribution. Dose and fluence profiles are computed using a
Monte Carlo simulation of a stopping-power-driven stochastic
differential equation.

The goal is to match a prescribed dose profile delivering uniform dose
to a target region between $3$ and $6\,\mathrm{cm}$ depth. A weighted
least-squares cost functional is minimised over nonnegative linear
combinations of the pencil beams, corresponding to intensity-modulated
inputs. The resulting weights define the optimal beam configuration,
from which dose and fluence are computed.

Figure~\ref{fig:1d-fluence} shows the optimised fluence in
depth-energy coordinates, along with the relative input spectrum and
resulting depth-dose curve. The modulation of Bragg peaks across the
beam bank produces a spread-out Bragg peak (SOBP) that conforms to the
target region.

\begin{figure}[h!]
  \centering
  \includegraphics[width=0.95\textwidth]{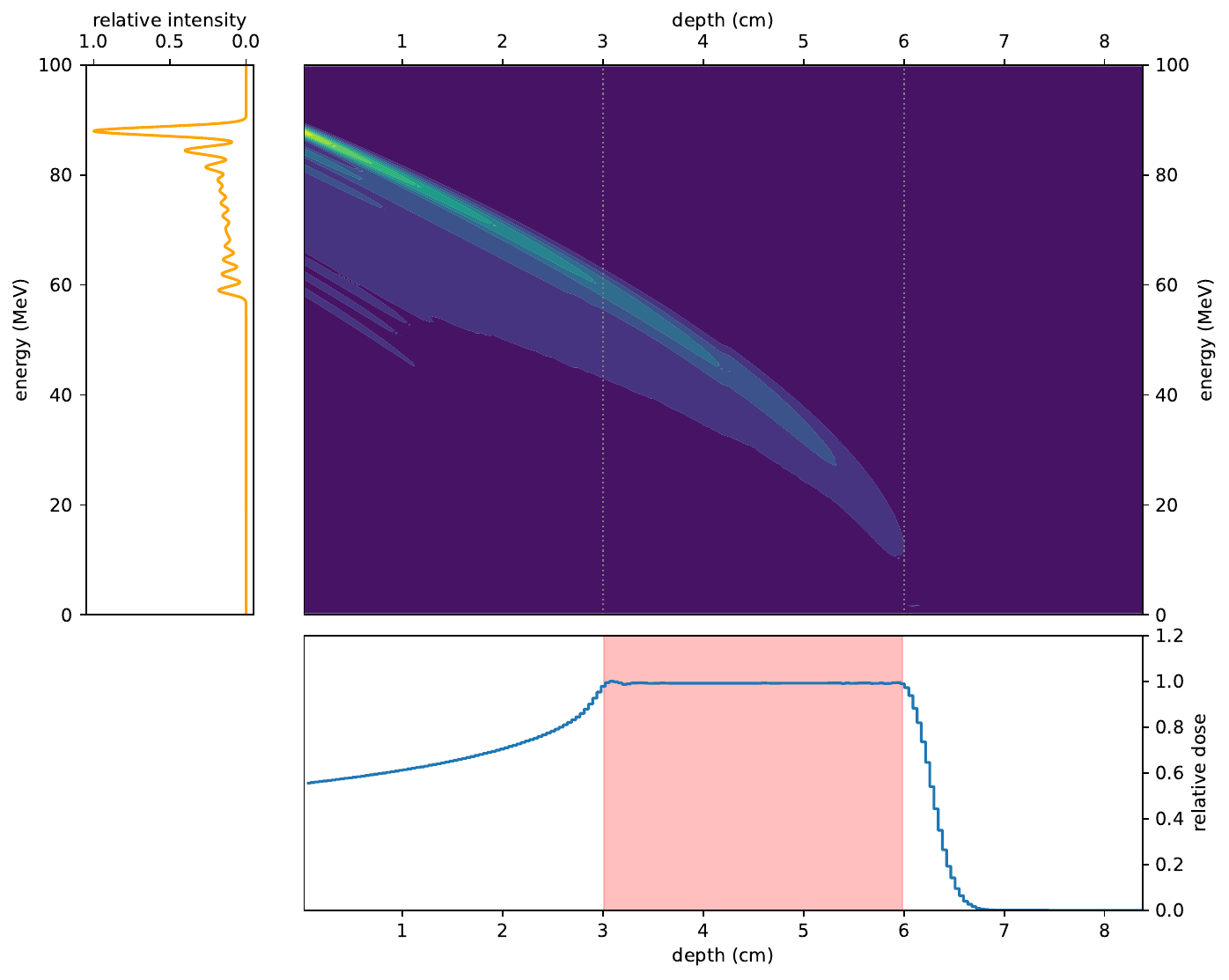}
  \caption{Numerical result of the 1D dose optimisation experiment.
    \textbf{Top left:} Relative energy spectrum of the optimised beam.
  \textbf{Bottom:} Resulting depth-dose curve.
  \textbf{Top right:} Fluence in depth-energy coordinates. The target region
  between $3$ and $6\,\mathrm{cm}$ is shown in red.}
  \label{fig:1d-fluence}
\end{figure}

\section{Conclusion}
\label{sec:conclusion}

This work develops a unified stochastic-deterministic framework for
modelling proton transport in radiotherapy. Starting from a stochastic
process describing individual proton tracks, we established its
connection to the Boltzmann--Fokker--Planck equation via variational
theory and dual semigroup structures. The resulting resolvent
formulation links expected occupation measures with deterministic PDE
solutions, providing a rigorous basis for interpreting dose as both a
stochastic functional and a variational quantity.

The adjoint relationship between forward and backward transport
enables a dual characterisation of fluence and dose, with implications
for uncertainty quantification, control, and optimisation. As a proof
of concept, we applied this framework to a robust treatment planning
problem, using stochastic forward models and deterministic
adjoint-based gradient computation. A simple numerical demonstration
illustrated how the theory underpins hybrid algorithm design.

Beyond optimisation, the stochastic-deterministic duality supports
theoretical analysis, variance reduction, and model reduction
strategies. It provides a coherent foundation for incorporating
uncertainty into transport-based methods, bridging probabilistic
simulation and PDE-based inference in a mathematically consistent way.

{\sloppy
\printbibliography}

\end{document}

%% file: definitions.tex
\newcommand{\leb}{\text{Leb}}

